%% file: localEstimates_rev_arxiv.tex
\documentclass[11pt]{amsart}
\usepackage{graphicx}
\usepackage{subfig}
\usepackage{tabularx}
\usepackage{wasysym,mathtools}

\usepackage{mathrsfs}     % selects Times Roman as basic font
\usepackage{helvet}         % selects Helvetica as sans-serif font
\usepackage{courier}        % selects Courier as typewriter font
\usepackage{type1cm}      % activate if the above 3 fonts are
\usepackage{color}
\usepackage{url}

\usepackage{geometry,calc,color}
\usepackage{amsmath,amssymb}
\usepackage{esint}

\usepackage{enumerate}
\usepackage{arydshln}
\usepackage{siunitx}
\usepackage{booktabs}
\usepackage{multirow}
\usepackage{tikz}
\usepackage{pgfplots}
\usepackage{enumitem}
\pgfplotsset{compat=1.9}

\newcommand*{\changesizes}{%
	\let\Huge=\huge
	\let\huge=\LARGE
	\let\LARGE=\Large
	\let\Large=\large
	\let\large=\normalsize
	\let\normalsize=\small
	\let\small=\tiny
}
\newcommand*{\switchsize}{%
	\changesizes
	\normalsize}

\newtheorem{theorem}{Theorem}[section]
\newtheorem{corollary}[theorem]{Corollary}
\newtheorem{lemma}[theorem]{Lemma}
\newtheorem{proposition}[theorem]{Proposition}
\theoremstyle{definition}

\newtheorem{assumption}[theorem]{Assumption}
\newtheorem{comment}[theorem]{Comment}
\theoremstyle{remark}
\newtheorem{remark}[theorem]{Remark}
\renewcommand{\theta}{\vartheta}

\DeclareMathOperator{\supp}{supp}
\numberwithin{equation}{section}

\newcommand{\roundPrecision}{2}

\newcommand{\blu}[1]{{\color{black} #1}} % blu
\newcommand{\rosso}[1]{{\color{black} #1}} %rosso
\long\def\NOTE#1{}
\begin{document}

\title[Interior Estimates for VEM]{Interior estimates for the Virtual Element Method}%

\author[S. Bertoluzza]{Silvia Bertoluzza}
\address{IMATI ``E. Magenes'', CNR, Pavia (Italy)}%
\email{silvia.bertoluzza@imati.cnr.it}%

\author[M. Pennacchio]{Micol Pennacchio}
\address{IMATI ``E. Magenes'', CNR, Pavia (Italy)}%
\email{micol.pennacchio@imati.cnr.it}%

\author[D. Prada]{Daniele Prada}
\address{IMATI ``E. Magenes'', CNR, Pavia (Italy)}%
\email{daniele.prada@imati.cnr.it}%

\date{\today}
\thanks{{This paper has been realized in the framework of ERC Project CHANGE, which has received funding from the European Research Council (ERC) under the European Union’s Horizon 2020 research and innovation programme (grant agreement No 694515), and was co-funded by the MIUR Progetti di Ricerca di Rilevante Interesse Nazionale (PRIN) Bando 2017 (grant 201744KLJL). The authors  are members of the Gruppo Nazionale Calcolo Scientifico-Istituto Nazionale di Alta Matematica (GNCS-INDAM)}}%

% ----------------------------------------------------------------
\begin{abstract} We analyze the local accuracy of the virtual element method. More precisely, we prove an error bound similar to the one holding for the finite element method, namely, that the local $H^1$  error in a interior subdomain is bounded by a term behaving like the best approximation allowed by the local smoothness of the solution in a larger interior subdomain plus the global error measured in a negative norm. 
\end{abstract}
\maketitle
% ----------------------------------------------------------------

\input{localEstimates_def}

\section{Introduction}
Besides its ability to handle complex geometries, one of the  features that contributed to the success of the finite element method as a tool for solving second order elliptic equations is their local behavior. Considering, to fix the ideas, the Poisson equation \[
- \Delta u = f, \quad \text{ in }\Omega \qquad u = 0,\quad \text{ on }\partial\Omega, 
\] the standard, well known, error estimate  provides, for the %$P^k$
order $k$ finite element method, an error bound of the form
\begin{equation*}
u \in H^{s+1}(\Omega), \quad 0 < s \leq k \qquad \Longrightarrow \qquad\| u - u_h \|_{1,\Omega} \leq C h^{s} | u |_{s+1,\Omega}
\end{equation*}
($u$ and $u_h$ denoting, respectively, the exact and approximate solutions, and $h$ the mesh size). When measured in a global norm, the error can be negatively affected  by the presence of even a few isolated singularities. However, since the early days in the history of such a method it is well known  that, if a solution with low overall regularity is locally smoother, say in a subdomain $\Omega_1 \subset\subset\Omega$, an asymptotically higher order of convergence can be expected in any domain $\Omega_0 \subset\subset \Omega_1$.
%\rosso{[Inserire citazione lavori Schatz-Wahlbin]}. 
More precisely, there exists an $h_0$ (depending on $\Omega_0$ and $\Omega_1$) such that, for $h < h_0$, the $H^1(\Omega_0)$ norm of the error can be bounded \blu{(\cite{NS}, see also \cite{SW,SW2})} as
\begin{equation}\label{NSbound} \|u - u_h
\|_{1,\Omega_0} \lesssim  h^{k} \|  u \|_{k+1,\Guno} + \| u - u_h \|_{1-k,\Omega},
\end{equation}
which, combined with an Aubin--Nitsche argument to bound the negative norm on the right hand side, yields, if $\Omega$ is sufficiently smooth,  an $O(h^k)$ bound for the error in the $H^1(\Omega_0)$ norm, provided $u \in H^{k+1}(\Omega_1)$, even when $u \not \in H^{k+1}(\Omega)$. This feature is particularly appealing, as it allows to take advantage of the local regularity of the solution, thus enabling the method to perform effectively. One can, for instance,  avoid the need  of refining the mesh, whenever possible singularities are localized far from a region of interest. 

It is therefore clearly desirable that new methods, aimed at  generalizing finite elements, retain this property. We focus here on the virtual element method, a discretization approach that generalizes finite elements to general polygonal space tessellations. Analogously to the finite element method, the virtual element discretization space  is continuously assembled from local spaces,
constructed element by element in such a way that polynomials up to order $k$ are included in the local space. Contrary to the finite element case, however, the functions in the space are not known in closed form but are themselves solution to a partial differential equation, which is, however, never solved in the implementation. In order to handle the discrete functions, these are instead split as the sum of an exactly computable polynomial part, and of a non polynomial part. Exact handling of the polynomial part alone turns out to be sufficient to guarantee  good approximation properties: to this end, the bricks needed for the solution of the problem at hand by a Galerkin approach (e.g. local contribution to the bilinear form and right hand side) are computed as a function of a set of unisolvent degrees of freedom, in a way that is locally exact for polynomials. The non polynomial part is instead handled by means of a {\em stabilization term}, which only needs to be spectrally equivalent to the bilinear form considered, resulting in a non conforming approximation. 
Since its introduction  in the early 2010s (see \cite{BeiraodaVeiga-Brezzi-Cangiani-Manzini-Marini-Russo:2013,VEMHitchhiker}), the virtual element method has gained the interest of the scientific community and has seen a rapid development, with numerous contributions aimed at the theoretical analysis of the method
(see e.g. \cite{
BeiraodaVeiga-Lovadina-Russo:2017,
BeiraodaVeiga-Brezzi-Cangiani-Manzini-Marini-Russo:2013,
Brenner-VEM,
Brenner-VEM-small-edges}
), its efficient implementation (see, e.g., \cite{
VEM-FETI-2D,
VEM-FETI-3D,Dassi-Schacchi-2020,Dassi-Scacchi-2020-2,Calvo2019}),  its extensions in different directions (see, e.g., \cite{VEM-curvo-nostro, 
VEM-general-mixed,
VEM-curved,
VEM-serendipity,Beiraoetal:hp:2016,VEM-Hdiv,Cernov-Marcati-Mascotto:2021}), and applications in different fields, such as fluid dynamics \cite{navier-stokes2d,stokes3d,Antonietti_VEM_Stokes,VEM_discrete_fracture}, continuum mechanics \cite{BEIRAODAVEIGA2015327,CHI2017148,beirao_elastic,
beirao_linear_elasticity,
VEM_3D_elasticity,wriggers2017efficient,
wriggers2016virtual}, electromagnetism \cite{BEIRAODAVEIGA2021} and others 
(\cite{Antonietti_VEM_Cahn,
	perugia_Helmholtz,
	ABPV:minsurf}).

\

%We aim, in this paper, 
In this paper, we aim at proving that the approximation by the virtual element method has good localization properties, similar to the ones displayed by the finite element method. More precisely, 
under suitable assumptions on the tessellation, we will prove that an estimate of the form \eqref{NSbound} also holds for the virtual element solution (see Theorem \ref{thm:main}). 
Under suitable assumptions on the domain $\Omega$ (the same needed for the analogous result in the finite element method),  this will imply that, provided $u\in H^{k+1}(\Omega_1)$, we have that $\| u - u_h \|_{1,\Omega_0} = O(h^{k})$, independently of the overall smoothness of the solution.
	
\

The paper is organized as follows. After presenting some notation and recalling how some inequalities do (or do not) depend on the shape and size of the elements (see Section \ref{sec:notation}), in Section \ref{sec:vem} we present the virtual element formulation we will be focusing on, and we will study the equation satisfied by the error. In Section \ref{sec:commutator} we will study how different linear and bilinear operators commute with the multiplication by a smooth weighting function. In Sections \ref{sec:negativeglobal}, \ref{sec:negativesmooth} and \ref{sec:negativelocal} we will provide bounds for the error in, global and local negative norms. In Section \ref{sec:main} we will prove the main result, namely Theorem \ref{thm:main}, and leverage it to obtain local error bounds (see Corollaries \ref{cor:finalpoly} and \ref{cor:finalsmooth}). \blu{In Section \ref{sec:enhanced} we briefly sketch an extension of the local error bounds to the so called {\em enhanced} version of  virtual element method \cite{EnhancedVEM}, which is often the one that can be found in actual implementations.} Finally, in Section \ref{sec:numerical}, we present some numerical results.	

Throughout the paper, we will write $A \lesssim B$ to indicate that $A \leq c B$, with $c$ independent of the mesh size parameters, and depending on the shape of the elements only through the constants $\gstar$ and $\guno$ in  the shape regularity \blu{A}ssumption \ref{geometry}. The notation $A \simeq B$ will stand for $A \lesssim B \lesssim A$.

\section{Notation and preliminary bounds}\label{sec:notation}

In the following we will use the standard notation for Sobolev spaces of both positive and negative index, and for the respective norms (see \cite{LionsMagenes}). Letting $\Omega \subset \mathbb{R}^2$ denote a bounded  polygonal domain, we will consider a family $\mathcal{F} = \{\Tess\}$ of  polygonal tessellations of $\Omega$, depending on a mesh size parameter $h$. We make the following assumption on the tessellations. 
\begin{assumption}\label{geometry} There exist constants $\gamma_0,\gamma_1 >0$ such that, letting $h_K$ denote the diameter of the polygon $K$, for all tessellation $\Tess \in \mathcal{F}$:
	\begin{enumerate}[label=(\alph*)]
		\item 	All polygons $\K \in \Tess$ are star shaped with respect to all points of a ball with center $\xK$ and radius $\rK$ with $\rK \geq \gstar \hK$;
		\item for all $K\in \Tess$	the distance between any two vertices of $\K$ is greater than $\guno \hK$.
	\end{enumerate}	
	Moreover, for the sake of notational simplicity, we assume that all tessellations are quasi-uniform, that is for all $K \in \Tess$, we have that $h_K \simeq h$.
\end{assumption}

%In the following we will write $A \lesssim B$ to indicate that $A \leq c B$, with $c$ independent of the mesh size parameters $h$ and $h_K$, and depending on the shape of the elements only through the constants $\gstar$ and $\guno$.

\

The following trace and Poincar\'e inequalities hold, with constants only depending on the two constants $\gamma_0$ and $\gamma_1$ (see, e.g., \cite{BeiraodaVeiga-Lovadina-Russo:2017,Brenner-VEM-small-edges}).

\subsection*{Trace inequalities} Under Assumption \ref{geometry}(a)
for all $v \in H^1(K)$, we have
\begin{gather}\label{trace}
\| v \|_{0,\partial K}^2 \lesssim \| v \|_{0,K} ( h_K^{-1} \| v \|_{0,K} + \| \nabla v \|_{0,K}), \\[2mm]
| v |_{1/2,\bK} \lesssim | \blu{v}  |_{1,K}, \label{trace2}\\[2mm]
\inf_{w \in H^1_0(K)} | v + w |_{1,K}\lesssim | v |_{1/2,\bK}. \label{trace3}
\end{gather}

\subsection*{Poincar\'e inequality} Under Assumption
\ref{geometry}(a), for all $u \in H^1(K)$, we have
\begin{equation}\label{poincare}
\inf_{q \in \mathbb{R}} \| u - q \|_{0,K} \lesssim h_K \| \nabla u \|_{0,K}.
\end{equation}

\

We have the following proposition, where the norm $\| \cdot \|_{-1,K}$ is defined as
\[
\| F \|_{-1,K} = \sup_{{v\in H^1_0(K)} \atop{v \not =0}} \frac{\langle F, v \rangle }{| v |_{1,K}}, \qquad \text{$\langle \cdot,\cdot \rangle$ denoting the duality product of $H^{-1}(K)$ and $H^1_0(K)$. }
\]

\begin{proposition}\label{prop:2.2}
Under Assumption
	\ref{geometry}(a), for all $u \in H^1(K)$, we have
		\[
		| v |_{1,K} \lesssim | v |_{1/2,\blu{\partial}K} + \| \Delta v \|_{-1,K}.
		\]
\end{proposition}

\begin{proof}
	We split $v$ as $v^H + v^0$ with $\Delta v^H = 0$ and $v^0 \in H^1_0(K)$. The splitting is stable with respect to the $H^1$ semi norm, that is we have
	\begin{equation}\label{splitting}
	\int_K | \nabla v |^2 = \int_K |\nabla  v^H + \nabla v^0 |^2 =  \int_K |\nabla  v^H |^2 + \int_K | \nabla v^0 |^2.
		\end{equation}
			Now,  for $w \in H^1_0(K)$ arbitrary, since $v^H = v$ on $\bK$, integrating by parts twice we can write 
		\begin{gather*} 
		| v^H |^2_{1,K} = 
		\int_{\bK} v^H   \nabla v^H \cdot\ n_K = \int_{\bK} (v + w) \nabla v^H \cdot \ n_K  = \int_{K} \nabla (v + w) \cdot \nabla v^H \leq | v + w |_{1,K} | v^H |_{1,K},
\end{gather*}
where $n_K$ denotes the outer unit normal to $\partial K$. Thanks to the arbitrariness of $w$, dividing both sides by $| v^H |_{1,K}$ and using \eqref{trace3} we obtain
\begin{equation} \label{boundvH}
| v^H |_{1,K} \leq \inf_{w\in H^1_0(K)} | v + w |_{1,K} \lesssim | v |_{1/2,\blu{\partial}K}.
\end{equation}
On the other hand, as $\Delta v^0 = \Delta v$, we can write
		\begin{equation*}
		| v^0 |^2_{1,K} = \int_K | \nabla v^0 |^2 = - \int_K v^0 \Delta v^0 = - \int_K v^0 \Delta v \leq | v^0 |_{1,K} \| \Delta v \|_{-1,K},
		\end{equation*}
		which, dividing both sides by $| v^0 |_{1,K}$, yields
		\begin{equation}\label{boundv0}
			| v^0 |_{1,K} \leq \| \Delta v \|_{-1,K}.
		\end{equation}
		By collecting \eqref{boundvH} and \eqref{boundv0} into \eqref{splitting} we obtain the desired bound.
	\end{proof}

\

 For $D$ being a polygonal element $K$ or an edge $e$ of $\Tess$, we let $\Poly{\ell}(D)$ denote the restriction to $D$ of the space of bivariate polynomials of order up to $\ell$. 

Under Assumption \ref{geometry}, we have the following polynomial approximation bounds 
\blu{\cite{DupontScott}}: for all $v \in H^t(K)$, $0 \leq s \leq t \leq \ell+1$
\begin{equation}\label{eq:polyapproxK}
\inf_{q \in \Poly{\ell}(K)} \| v - q \|_{s,K} \lesssim h_K^{t-s} | v |_{t,K},
\end{equation}
and, letting $h_e$ denote the length of the edge $e$, for all $v \in H^t(e)$, $0 \leq s \leq t \leq \ell+1$
\begin{equation}\label{eq:polyapproxe}
\inf_{q \in \Poly{\ell}(e)} \| v - q \|_{s,e} \lesssim h_e^{t-s} | v |_{t,e}.
\end{equation}

Moreover, the following inverse inequalities for polynomial functions hold: for all $q \in \Poly{\ell}(K)$ and $q \in \Poly{\ell}(e)$, and all $0 \leq s \leq t$
\begin{equation}\label{eq:polyinverse}
\| q \|_{t,K} \lesssim h_K^{-(t-s)} \| q \|_{s,K}, \qquad \| q \|_{t,e} \lesssim h_e^{-(t-s)} \| q \|_{s,e}.
\end{equation}

 %  In the following, for $D\subset \mathbb{R}^2$ bounded Lipschitz domain we let $H^m(\Omega)$ denote the standard Sobolev space of square integrable functions with square integrable derivatives up to order $m$, endowed with the standard Sobolev norm and seminorm
%\[
%\| f \|_{m,D}^2 = \sum_{\ell=0}^{m} \sum_{|\alpha|=\ell}  \int_D  | \partial f /\partial x^\alpha |^2, \qquad | f |_{m,D}^2  \sum_{|\alpha|=m} = \int_D  | \partial f /\partial x^\alpha |^2.
%\]
%We will denote by $H^m
%For positive $m$, $H^{-m}(D)$ will, as usual, denote the dual of the space $H^m_0(D)

\section{The Virtual Element Method}\label{sec:vem}
In order to introduce the notation, let us review the definition of the simple form of the virtual element method that we are going to consider.
Letting $\Omega \subset \mathbb{R}^2$, we focus on the following model problem:
\begin{equation}\label{Pb:strong}
-\Delta u = f, \quad \text{ in }\Omega, \qquad u=0 \quad\text{ on }\partial\Omega,
\end{equation}
which, in weak form, rewrites as: find $u \in H_0^1(\Omega)$ such that 
\begin{equation}\label{Pb:weak}
a(u,v) = \int_{\Omega} f v, \quad \forall v \in H^1_0(\Omega), \qquad \text{ with }\quad a(u,v) = \int_{\Omega}\nabla u\cdot \nabla v.
\end{equation}
Let $\Tess \in \family$ denote a  tessellation of $\Omega$ in the family $\family$. 
For reasons that will be clear in the following, we consider a form of the Virtual Element discretization where we allow different approximation orders on the boundary and in the interior of the elements \cite{beirao2020different_order}. As usual, for all $K \in \Tess$ we let
\[
\BoK = \{ v \in C^0(\bK):\ v|_e \in \mathbb{P}_k(e)\ \text{ for all edges $e$ of $K$} \}.
\]
The local element space $\VEMK$ is defined \blu{(see \cite{beirao2020different_order})} as 
\[
\VEMK = \{ v \in H^1(K): \ v|_{\bK} \in \BoK, \text{ and } \Delta v \in \mathbb{P}_{m}(K) \}.
\]
We  assume that $\max\{0,k - 2\} \leq m \leq k$. The case $m = k-2$ corresponds to the simplest form of the virtual element method, as introduced in \cite{BeiraodaVeiga-Brezzi-Cangiani-Manzini-Marini-Russo:2013}. 

\

The following inverse inequality holds for all $v_h \in \VEMK$ (see \cite{Cangiani-Georgoulis-Pryer-Sutton:2016})
	\begin{equation}\label{inverseVEM}
	\| \Delta v_h \|_{0,K} \lesssim h_K^{-1} \| \nabla v_h \|_{0,K} .
	\end{equation}

\

\begin{remark}
For the sake of simplicity, we do not explicitly include in our analysis the simplest lowest order VEM space
\[
V^1_{-1}(K) = \{
\vh \in H^1(K):\ \vh|_{\partial K} \in \mathbb{B}_1(\partial K)\ \text{ and }\ \Delta \vh = 0\},
\]
which can be tackled by the same kind of argument but which would require a separate treatment, in particular when dealing with the terms involving the approximation of the right hand side. 
\end{remark}

\

The global discretization space $\Vh$ is defined as 
\begin{equation}\label{defVh}
\Vh = \{ v \in \blu{H_{0}^1}(\Omega): \ v|_K \in \VEMK \ \text{ for all $K \in \Tess$} \}.
\end{equation}
Letting 
\[
H^1(\Tess) = \{ u \in L^2(\Omega): u|_K \in H^1(K) \ \text{ for all $K \in \Tess$} \}
\]
denote the space of discontinuous piecewise $H^1$ functions on the tessellation $\Tess$, which we endow with the seminorm and norm
\[ | v |_{1,\Tess}^2 = \sum_{K\in \Tess} | v |_{1,K}^2, \qquad \| v \|_{1,\Tess} = \| v \|_{0,\Omega} + | v |_{1,\Tess},\]
we also introduce the discontinuous global discretization space
\[
\Vhstar = \{ v \in L^2(\Tess): \ v|_K \in \VEMK \ \text{ for all $K \in \Tess$} \}.
\]

As usual, we introduce the local projector $\PinablaK: H^1(K) \to \mathbb{P}_k(K)$ defined as
\[
\int_K \nabla(\PinablaK v - v) \cdot \nabla q = 0, \quad \forall q \in  \mathbb{P}_k(K), \quad \int_{K} (\PinablaK v - v) = 0.
\]

As $\PinablaK$ preserves polynomials of order up to $k$, the following proposition is not difficult to prove, the bound on $| \cdot |_{1,K}$ being a direct consequence of \eqref{eq:polyapproxK} and the bound on $\| \cdot \|_{0,K}$ being proved by an Aubin-Nitsche duality argument.
\begin{proposition}\label{prop:3.2}
	Let $v \in H^{1+s}(K)$, $0\leq s \leq k$. Then we have
	\[
	\| v - \PinablaK v \|_{0,K} + h_K | v - \PinablaK v |_{1,K} \lesssim h_K^{1+s} | v |_{1+s,K}.
	\]
\end{proposition}

\

Letting
\[
\Pstar{k} = \{ q \in L^2(\Omega): q|_K \in \mathbb{P}_k(K),\ \text{ for all $K \in \Tess$} \}
\]
denote the space of discontinuous piecewise polynomials of order up to $k$ defined on the tessellation $\Tess$, we let $\Pinabla: H^1(\Tess) \to \Pstar{k}$ be defined as
\[
\Pinabla v |_K := \PinablaK (v|_K).
\]

\

The discrete bilinear form $a_h: \Vh\times\Vh \to \mathbb{R}$ is defined as
\[
a_h(w_h,v_h) = \int_{\Tess}\nabla\Pinabla w_h \cdot \nabla \Pinabla v_h + s_h ((\Id-\Pinabla) w_h, (\Id-\Pinabla) v_h),
\]
where $\Id$ is the identity of $H^1(\Tess) $   and where, for shortness,  here and in the following we use the conventional notation 
\[
\int_{\Tess} X := \sum_{K\in \Tess} \int_K X.
\]
The stabilization bilinear form $s_h: \Vhstar \times \Vhstar$ is defined as the sum of local contributions
\[
s_h(w_h,v_h) = \sum_K \sK(w_h|_K,v_h|_K),
\]
where we assume, as usual, that, for all $K \in \Tess$, the local stabilization bilinear form $\sK: \VEMK \times \VEMK \to \mathbb{R}$ satisfies
\begin{gather}
\sK(w_h,v_h) \lesssim | w_h |_{1,K} | v_h |_{1,K} \qquad \forall \vh, \wh \in \VEMK, \label{conds1}\\[1mm]
\sK(\wh,\wh) \gtrsim | w_h |^2_{1,K} \qquad \forall \wh \in \VEMK \cap \ker \PinablaK. \label{conds2}
\end{gather}
In particular, for all $\vh \in \VEMK$, \eqref{conds1} and \eqref{conds2} yield
 \begin{equation} \label{stabtermcondition}
| \vh - \PinablaK \vh |^2_{1,K} \simeq \sK(\vh - \PinablaK \vh, \vh - \PinablaK \vh).
 \end{equation}
We let 
\[
a_h^K (\wh,\vh) = \int_{K} \nabla\PinablaK w_h \cdot \nabla \PinablaK v_h + \sK((\Id-\PinablaK) w_h, (\Id-\PinablaK) v_h ) 
\]
denote the local counterpart of the bilinear form $a_h$.  We recall that, thanks to \eqref{conds1} and \eqref{conds2}, for all $v \in \Vh$ we have that
\begin{equation}\label{3.5b}
a^K(v_h,v_h) = \int_K \nabla v_h \cdot \nabla v_h \simeq a^K_h(v_h,v_h).
\end{equation}

%We let $\dofi{\cdot}$, $i=1, \cdot , \dim(\VEMK)$ denote the $\dim(\VEMK)$ unisolvent degrees of freedom identifying the elements of $\VEMK$, which are naturally split in two sets: boundary degrees of freedom and interior degrees of freedom. The set of boundary degrees of freedom is in turn split in vertex and edge degrees of freedom.

\

We next let $\Pizero{m}: L^2(\Omega) \to \Pstar{m}$ denote the $L^2(\Omega)$ orthogonal projection  onto the space $\Pstar{m}$
 of discontinuous piecewise polynomials of order at most $m$, and we let  
\[
\fh = \Pizero{m} f
\]
so that for all $\vh \in \Vh$
\begin{equation}\label{deffh}
\int_{\Omega} \fh \vh = \int_{\Omega} \Pizero{m}f\, \vh = \int_{\Omega} f\, \Pizero{m} \vh.
\end{equation}

\

The virtual element solution to Problem \eqref{Pb:strong} is obtained by solving the following discrete problem: find $\uh \in \Vh$ such that for all $\vh\in \Vh$
\begin{equation}\label{pb:discrete}
a_h(\uh,\vh) = \int_{\Omega} f_h \vh.
\end{equation}

\subsection{Extension of the discrete operators to $H^1(\Tess)$}
To carry out the forthcoming analysis, it will be convenient to extend some of the above operators, which are defined on the discrete space $\Vh$, to the whole $H^1(\Tess)$. To this aim, we introduce projectors $\tPinablaK: H^1(K) \to \VEMK$ and $\QnablaK: H^1(K) \to \VEMK$ defined as
\begin{gather}
%
%	
%	\[
%	\int_K \nabla(\PinablaK v - v) \cdot \nabla q = 0, \quad \forall q \in  \mathbb{P}_k(K), \quad \int_{K} (\Pinabla v - v) = 0.
%	\]
	\int_K \nabla (\tPinablaK v -  v ) \cdot \nabla w_h= 0, \quad \forall w_h \in \VEMK, \qquad \int_{K} (\tPinablaK v - v) = 0, \label{deftPinablaK}\\
%\int_K \nabla \tPinablaK v \cdot \nabla w_h = \int_K \nabla  v \cdot \nabla w_h, \quad \forall w_h \in \VEMK, \qquad \int_{K} (\tPinablaK v - v) = 0, \label{deftPinablaK}\\
\QnablaK v = \tPinablaK v - \PinablaK v. \label{defQnablaK}
\end{gather}
Observe that we have $\tPinablaK \circ \PinablaK  = \PinablaK \circ \tPinablaK = \PinablaK$ and $\tPinablaK \circ \QnablaK = \QnablaK \circ \tPinablaK  = \QnablaK$. 
Also the projectors  $\tPinablaK$ and $\QnablaK$ can be assembled, element by element, to global projectors into the space $\Vhstar$ of discontinuous virtual element functions.
More precisely we define $\tPinabla: H^1(\Tess) \to \Vhstar$ and $\Qnabla: H^1(\Tess) \to \Vhstar$ as
\[
 \tPinabla v |_K = \tPinablaK (v|_K), \qquad \Qnabla v |_K = \QnablaK (v|_K).
\]
With this notation, we can extend the bilinear form $a_h:\Vh \times \Vh \to \mathbb{R}$ to a bilinear form $a_h: H^1(\Tess) \times H^1(\Tess) \to \RR$, defined as
\[
a_h(u,v) = \int_{\Tess} \nabla (\Pinabla u)\cdot \nabla (\Pinabla v) + s_h(\Qnabla u, \Qnabla v).
\]

\
We remark that, thanks to \eqref{3.5b}, we have
\begin{equation}\label{coerch}
a(u,u)  \simeq \int_{\Tess}|\nabla (\Id - \tPinabla) u|^2 + a_h(\tPinabla u,\tPinabla u), \qquad \forall u \in H^1(\Omega).\end{equation}

\

As $\QnablaK q = 0$ for all $q \in \Poly{k}(K)$, we easily have the following proposition, where the $H^1(K)$ seminorm bound stems from the best polynomial approximation and the $L^2(K)$ bound is obtained by a Poincar\'e inequality, as, by definition, $\QnablaK f$ is average free.
\begin{proposition}\label{boundQnablaK}
	Let $f \in H^{1+s}(K)$, $0 \leq s \leq k$. Then
	\[
	\| \QnablaK f \|_{0,K} + h_K | \QnablaK f |_{1,K} \lesssim h_K^{1+s} | f |_{1+s,K}.
	\]
	\end{proposition}

\subsection{Error equations}
Letting $u \in H^1_0(\Omega)$ and $\uh \in \Vh$ respectively denote the solutions of Problem \eqref{Pb:strong} and \eqref{pb:discrete}, we can now write two error equations, satisfied by $u - u_h$. Indeed, for $v \in H^1(\Omega)$ arbitrary we have
\begin{multline*}
a_h(u,v) = \int_{\Tess}\nabla \Pinabla u\cdot\nabla \Pinabla v + s_h (\Qnabla u, \Qnabla v) \\ = \int_{\Omega} \nabla u \cdot \nabla v - \int_{\Tess}\nabla \Qnabla u \cdot \nabla \Qnabla v - \int_{\Tess}\nabla (\Id-\tPinabla) u \cdot \nabla (\Id-\tPinabla) v + s_h (\Qnabla u,\Qnabla v),
\end{multline*}
whence
\begin{equation}\label{eq:conformity}
a(u,v) - a_h(u,v) =
\Dh(u,v) + \int_{\Tess}\nabla (\Id-\tPinabla) u \cdot \nabla (\Id-\tPinabla) v,
\end{equation}
where $\Dh: H^1(\Tess) \times H^1(\Tess) \to \mathbb{R}$ is defined as
\begin{equation}\label{defDh}
\Dh (w,v) = \sum_K \DK (u,v)
\end{equation}
with 
\[
\DK(w,v) = \int_{K} \nabla \QnablaK w \cdot \nabla \QnablaK v - s_h(\QnablaK w,\QnablaK v).
\]

\

Then, letting $\df = f - \fh$, the error $u - u_h$ satifies, for all $\vh \in \Vh$, 
\begin{multline*}
a_h(u - u_h, \vh) = a(u, \vh) + a_h(u,\vh) - a(u,\vh) - a_h(\uh,\vh)  \\= \int_{\Omega} f \vh +  a_h(u,\vh) - a(u,\vh) - \int_{\Omega}\fh \vh,
\end{multline*}
finally yielding
the following error equation
\begin{equation}\label{eq:error}
a_h ( u - u_h , \vh) = \int_{\Omega} \df\, \vh - \Dh ( u , \vh), \qquad \forall \vh \in \Vh.
\end{equation}
Combining \eqref{eq:error} and \eqref{eq:conformity} we  also have
\begin{equation}\label{othererrorequation}
a(u-u_h,\vh)  = \Dh(u-\uh,\vh) +
\int_{\Omega}\df\, \vh - \Dh(u,\vh), \qquad \forall \vh \in \Vh.
 \end{equation}

\

Using Proposition \ref{boundQnablaK}, we easily see that the following proposition holds.
\begin{proposition}\label{boundfurbo} Let $w \in H^{s+1}(K)$, $v \in H^{t+1}(K)$ with    % $0 \leq t,s \leq k$
	 $t \geq 0, s \leq k$. Then we have 
	\begin{gather}
	| \DK(w,v) | \lesssim h_K^{s+t} | w |_{s+1,K} | v |_{t+1,K}.\label{errorDeltaK}
	\end{gather}
\end{proposition}

Moreover,  the following proposition provides an a priori  bound for the operator $\df$ appearing at the right hand side of the error equation.
\begin{proposition}\label{boundrhs} Let $f \in H^{r}(K)$, $v \in H^{t+1}(K)$ with $0 \leq r \leq m+1$,  $0 \leq t \leq m$. Then we have 
	\begin{gather}
	\int_K \df\, v  \lesssim
	h_K^{r+t+1} | f |_{r,K} | v |_{t+1,K}. \label{errorrhs}
	\end{gather}
\end{proposition}
\begin{proof}
We have
	\begin{multline}\label{eq:3.20}
	\int_K (f - \fh) v = \int_K (f - \pim f) (v - \Pizero{m} v) \lesssim \| (\Id -\Pizero{m}) f \|_{0,K} \| (\Id -\Pizero{m}) v \|_{0,K}\\[2mm]
	\lesssim
	h_K^{r } | f |_{r,K} h^{t+1}_K | v |_{t+1,K},
	\end{multline}
	which is the desired result. 
\end{proof}

Remark that using the above bounds, in combination with the error equation and an approximation estimate (see e.g. \eqref{interperror}), allows to retrieve the following \blu{(essentially well known, see \cite{BeiraodaVeiga-Brezzi-Cangiani-Manzini-Marini-Russo:2013, beirao2020different_order})} bound on the error $u - \uh$: if the solution $u$ and the source term $f$ of Problem \eqref{Pb:weak} satisfy, respectively, $u \in H^{r+1}(\Omega)$, %$0 \leq r$ 
$r \geq 0$ and $f \in H^{\rho}(\Omega)$, $\rho \geq 0$ then it holds that
\begin{equation}\label{stimaerrorestandard}
\| u - \uh \|_{1,\Omega} \lesssim h^{\min\{r,k\}} | u |_{1+r,\Omega} + h^{\min\{\rho,m+1\} + 1} | f |_{\rho,\Omega}.
\end{equation}
Remark that, as $f = -\Delta u$, we have that $\rho \geq r-1$. Moreover, by construction $m \geq k-2$. Then the second term on the right hand side, deriving from the approximation of the source term, 
 is asymptotically dominated by the first term, namely $h^{\min\{r,k\}} \| u \|_{1+r,\Omega}$. 
 
% \NOTE{
%\begin{multline}
%	\| u_I - u_h \|_{1,\Omega}^2 \lesssim a_h(u_I - u,u_I - u_h) + a_h(u - u_h,u_I - u_h) \lesssim \| u_I - u \|_{1,\Omega} \| u_I - u_h \|_{1,\Omega} + \int_\Omega \delta_f (u_I - u_h)
%+ \Delta_h(u,u_I - u_h)
%\end{multline} 
%}

\section{Commutator properties for  the VEM space}\label{sec:commutator}
The local bounds we aim at proving will involve multiplying different quantities by smooth weights. A key role will be played by the error resulting from commutating the action of such weights with different operators appearing in the definition and analysis of the VEM method. To analyze such errors, we start by introducing a local quasi-interpolation operator similar to the one proposed in  \cite{VEM-curved} and defined as follows. Given $v \in H^1(K)$ with $v|_{\bK} \in C^0(\bK)$ and $\Delta v \in L^2(K)$, we let $\IK v = v_h \in \VEMK$ be defined by 
\[
v_h |_{\bK} = I_K v, \qquad \Delta v_h = \pim \Delta v,
\]
where $I_K : C^0(\bK) \to \BoK$ denotes the edge by edge  interpolation operator with, as interpolation nodes, the nodes of the $k+1$ points Gauss-Lobatto quadrature formula, and where, by abuse of notation, we let $\pim : L^2(K) \to \Poly{m}(K)$ denotes the $L^2$ orthogonal projection. As  $\IK : H^2(K) \to \VEMK \subset H^1(K)$ is bounded and it preserves the polynomials of degree $k$,  we can see that, for $v \in H^{1+s}(K)$, $1\leq s \leq k$ we have
\begin{equation}\label{interperror}
\| v - \IK v \|_{0,K} + h_K | v - \IK v |_{1,K} \lesssim h_K^{s+1} | v |_{s+1,K}.
\end{equation}

\

Let now $\weight \in C^\infty(\bar\Omega)$ be a smooth weight function. Observe that, for $K \in \Tess$, we can split $\weight|_K$ as
\begin{equation}\label{splittingweight}
\weight|_K = \bar\weight^K + \resto, \quad \text{with}\quad 
\bar \weight^K \in \mathbb{R}, \quad \| \resto \|_{0,\infty,K} \lesssim  h_K\| \weight \|_{1,\infty,\Omega}, \quad \| \resto \|_{2,\infty,K} \lesssim \| \weight \|_{2,\infty,\Omega},
\end{equation}
(we can for instance take $\bar \weight^K = \weight(x^K)$, $x^K$ being the barycenter of $K$).

\
 
We have the following lemma, which we prove by an approach similar to the one in \cite{bertoluzza1999discrete}.
\begin{lemma}\label{lem:disccom}
For all $K \in \Tess$, for all $v \in H^{s+1}(K)$, $1 \leq s \leq k$, for all $v_h \in \VEMK$,  it holds
	\begin{equation}\label{disccom1}
	| \weight (v + \vh) - \IK(\weight( v + \vh)) |_{1,K} \lesssim h_K^s | v |_{s+1} + h_K \| v + \vh \|_{1,K},
	\end{equation}
	the implicit constant in the inequality depending on $\| \weight \|_{2,\infty,\Omega}$.
%\COMMENT{$s \geq 1$ serve perche interpoliamo e quindi le funzioni devono stare in $H^2$}
\end{lemma}

\begin{proof} 
Let $v_h \in \VEMK$. As $\IK \vh = \vh$ and as $\IK$ is a linear operator, we have 
\[
\weight \vh - \IK(\weight \vh) = \resto \vh - \IK(\resto \vh),
\]
with $\resto$ given by \eqref{splittingweight}.
Then, using Proposition \ref{prop:2.2} we can write
\begin{multline}\label{comm0}
| \weight \vh - \IK(\weight \vh) |_{1,K} \lesssim | \weight \vh - \IK(\weight \vh) |_{1/2,\bK} + \| \Delta(\weight \vh - \IK(\weight \vh)) \|_{-1,K}\\
= | \resto \vh - \IK(\resto  \vh) |_{1/2,\bK} + \| \Delta(\resto  \vh - \IK(\resto  \vh)) \|_{-1,K}.
\end{multline}
We separately bound the two terms on the right hand side of \eqref{comm0}, starting from the first one.	We remark that, on $e$ edge of $K$ we have that $\resto  \vh - \IK(\resto \vh) \in H^1_0(e) \subseteq H^{1/2}_{00}(e)$, where, we recall, $H^{1/2}_{00}(e)$ can be defined as the space of those functions $v \in L^2(e)$ such that setting  $\widetilde v = v$ in $e$ and $\widetilde v = 0$ in $\bK \setminus e$ it holds that $\widetilde v \in H^{1/2}(\bK)$. We endow $H^{1/2}_{00}(e)$ with the norm $\| v \|_{H^{1/2}_{00}(e)} = | \widetilde v |_{1/2,\bK}$. We recall that $H^{1/2}_{00}(e)$ is the interpolation space of exponent $1/2$ with respect to the interpolation couple $(L^2(e),H^1_0(e))$.
	Then, using a standard interpolation bound (see \cite{Triebel}) 
	and \eqref{interperror}, we can write
	\begin{multline}\label{eq:comm1}
	| \resto  \vh - \IK(\resto  \vh) |_{1/2,\bK} \lesssim \sum_\blu{\text{edges $e \subset \bK$}} \| \resto \vh - \IK(\resto  \vh) \|_{H^{1/2}_{00}(e)} \\
	\lesssim 
	\sum_\blu{\text{edges $e \subset \bK$}}  \| \resto  \vh - \IK(\resto \vh) \|_{0,e}^{1/2}  | \resto \vh - \IK(\resto  \vh) |_{1,e}^{1/2} \\ \lesssim
	\sum_\blu{\text{edges $e \subset \bK$}}  h_e^{1/2} | \resto \vh - \IK(\resto  \vh) |_{1,e}.
	\end{multline}
%		\COMMENT{
%			The scaling of the interpolation bound is correct. We have
%			\[
%			\| \cdot \|_{H^{1/2}_{00}(e)} \simeq \| \widehat{\cdot} \|_{H^{1/2}_{00}(\widehat e)} \lesssim \|  \widehat{\cdot} \|_{0,\widehat e} | \widehat{\cdot} |_{1,\widehat e} \simeq h_e^{1/2} h_e^{-1/2} \|  {\cdot} \|_{0,e} | {\cdot} |_{1,e}.
%			\]
%		}

 Now, using \eqref{eq:polyapproxe} and \eqref{eq:polyinverse}, we can write
 \begin{multline*}
| \resto \vh - \IK(\resto \vh) |_{1,e} 
 \lesssim h_e | \resto v_h |_{2,e} 
 \lesssim h_e (
 \| \resto \|_{0,\infty,e} | v_h |_{2,e} + \| \resto \|_{2,\infty,e} \| v_h \|_{1,e}
 )\\[2mm]
 \lesssim h_e  ( h_e | \vh |_{2,e} + \| \vh \|_{1,e} ) \lesssim h_e \| \vh \|_{1,e} \lesssim h_e^{1/2} \| \vh \|_{1/2,e},
 \end{multline*}
	which yields (we recall that, by Assumption \ref{geometry}(b), $h_e \simeq h_K$)
	\begin{equation}\label{eq:comm2}
		|\resto \vh - \IK(\resto \vh) |_{1/2,\bK} \lesssim \sum_\blu{\text{edges $e \subset \bK$}}  h_e \| \vh \|_{1/2,e} \lesssim h_K \| \vh \|_{1/2,\bK} \lesssim h_K \| v_h \|_{1,K}.
	\end{equation}
	As far as the second term at the right hand side in \eqref{comm0} is concerned, we have 
	\begin{equation*}
	\Delta(\resto \vh) = \vh \Delta \resto + 2\, \nabla\resto\cdot\nabla \vh + \resto \Delta \vh,
	\end{equation*}
	as well as
	\[
	\Delta (\IK (\resto \vh)) = \pim (\Delta(\resto \vh) ) = \pim (\vh \Delta \resto ) + 2\, \pim (\nabla\resto\cdot\nabla \vh) + \pim (\resto \Delta \vh),
	\]
	and then, by \blu{triangle} inequality,
	\begin{multline*}
	 \| \Delta(\resto \vh - \IK(\resto \vh)) \|_{-1,K}  \lesssim \| \vh \Delta \resto - \pim (\vh \Delta \resto ) \|_{-1,K} \\[1.8mm]+
	\|  \nabla\resto\cdot\nabla \vh - \pim ( \nabla\resto\cdot\nabla \vh) \|_{-1,K} +
	\|  \resto \Delta \vh - \pim  (\resto \Delta \vh) \|_{-1,K}.
	\end{multline*}
	\NOTE{It is in the previous step that, if we use the enhancement within the space, it does not work, at least, not directly, as for the enhanced space we cannot define the interpolation by asking that $\Delta(\IK v) = \pim(\Delta v)$. Probably, the remedy lies in a more clever definition of the interpolation operator.}
	We can bound the three terms by using a standard duality argument, which allows to bound the $H^{-1}(K)$ norm of any average free function 
	with $h_K$ times its $L^2$ norm, and we obtain
	\begin{multline*}
	\| \vh \Delta \resto - \pim (\vh \Delta \resto ) \|_{-1,K}  \lesssim h_K \| \vh \Delta \resto - \pim (\vh \Delta \resto ) \|_{0,K} \\
	\lesssim h_K \| \vh \Delta \resto \|_{0,K} \lesssim h_K \| \resto \|_{2,\infty,K} \| \vh \|_{0,K},	\end{multline*}
	as well as
	\begin{multline*}		\|  \nabla\resto\cdot\nabla \vh - \pim ( \nabla\resto\cdot\nabla \vh) \|_{-1,K} \lesssim h_K 	\|  \nabla\resto\cdot\nabla \vh - \pim ( \nabla\resto\cdot\nabla \vh) \|_{0,K}\\[2mm] \lesssim h_K \| \nabla\resto\cdot\nabla \vh \|_{0,K} \lesssim h_K \| \resto \|_{1,\infty,K } \| \nabla \vh \|_{0,K},  	\end{multline*}
	and, using \eqref{splittingweight} as well as the inverse inequality \eqref{inverseVEM},
	\begin{multline*}
		\|  \resto \Delta \vh - \pim  (\resto \Delta \vh) \|_{-1,K} \lesssim h_K 	\|  \resto \Delta \vh - \pim  (\resto \Delta \vh) \|_{0,K} \lesssim h_K	\|  \resto \Delta \vh \|_{0,K} % \lesssim 
		\\[2mm]
	 \lesssim	h_K \| \resto \|_{0,\infty,K} \| \Delta \vh \|_{0,K} \lesssim h_K^2 \| \Delta \vh \|_{0,K} \lesssim h_K \| \nabla v_h\|_{0,K},
				\end{multline*}
	finally			yielding
		\begin{equation}\label{eq:comm3}		 \| \Delta(\resto \vh - \IK(\resto  \vh)) \|_{-1,K} \lesssim h_K \| \vh \|_{1,K}.\end{equation}
%		
%		\COMMENT{Here the switch between seminorms and norms is because there are integrals on $K$, so the norms are unscaled.}
		
Using \eqref{eq:comm2} and \eqref{eq:comm3} in \eqref{comm0} yields
\begin{equation}\label{eq:comm4}
| \weight v_h - \IK (\weight v_h) |_{1,K} \lesssim h_K \| \vh \|_{1,K}.
\end{equation}

		\
		
		Let now $v \in H^{s+1}(K)$. Adding and subtracting $\PinablaK v$, using \eqref{interperror} and \eqref{eq:comm4}, then adding and subtracting $v$ and using the polynomial approximation bound \eqref{eq:polyapproxK}, we have	
		\begin{multline}\label{disccom2}
		| \weight (v+\vh) - \IK (\weight (v+\vh) ) |_{1,K} \\[2mm]
		\leq | \weight(v - \PinablaK v) - \IK(\weight(v - \PinablaK v)) |_{1,K} + | \weight(\PinablaK v + v_h) - \IK(\weight(\PinablaK v + v_h)) |_{1,K} \\[2mm] \lesssim 
		h_K | \weight (v - \PinablaK v) |_{2,K} + h_K \| \PinablaK v + v_h \|_{1,K} \lesssim 	h_K \| v - \PinablaK v\|_{2,K} + h_K \| \PinablaK v + v_h \|_{1,K}
		\\[2mm] \lesssim 	h_K \|  v - \PinablaK v \|_{2,K}  + h_K \| \PinablaK v - v \|_{1,K} + h_K \| v + \vh \|_{1,K}\\[2mm]
		\lesssim h_K^s | v |_{s+1,K} + h_K \| v + \vh \|_{1,K},
		\end{multline}
		which concludes the proof.
\end{proof}

\

As for all $w \in H^1(K)$ it holds that
\[
| w - \tPinablaK w |_{1,K}\, \blu{=} \inf_{w_h \in \VEMK} | w - w_h |_{1,K},
\]
we immediately have the following corollary.
\begin{corollary}\label{cor:disccom} For all $K \in \Tess$, for all $v \in H^{s+1}(K)$, $1 \leq s \leq k$, for all $v_h \in \VEMK$,  it holds 
	\begin{equation}
	\label{disccomtPinablaK}| \weight(v + \vh) - \tPinablaK(\weight(v+\vh)) |_{1,K} \lesssim  
	h_K^s | v |_{s+1} + h_K \| v + \vh \|_{1,K},
	\end{equation}
	the implicit constant in the inequality depending on $\| \weight \|_{2,\infty,\Omega}$.
\end{corollary}

\

\begin{remark}
	The bound \eqref{eq:comm4}, valid for all $\vh \in \VEMK$, is the so called {\em discrete commutator property} of the Virtual Element space. It implies that $\IK : \weight \VEMK \to \VEMK$ is bounded in $H^1(K)$. In particular, we have
	\begin{equation}\label{stabilityIK}
	| \IK(\omega v_h) |_{1,K} \lesssim | v_h |_{1,K}, \qquad \forall v_h \in \VEMK .
	\end{equation}
\end{remark}

	\
	
	Another commutativity bound that will play a role later on,  is the following lemma.
		\begin{lemma}\label{lem:4.4} For all $K \in \Tess$  it holds, for all $v \in H^1(K)$, 
			\[
			| \QnablaK(\weight v) - \weight \QnablaK v |_{1,K} \lesssim h_K \| v \|_{1,K}
			\]
			the implicit constant in the inequality depending on $\| \weight \|_{2,\infty,\Omega}$.		
		\end{lemma}
		
		\begin{proof}  Once again we use the splitting \eqref{splittingweight}, and, by the linearity of $\QnablaK$, we have
			\[
			\QnablaK(\weight v) - \omega \QnablaK v = \QnablaK(\resto v) - \resto \QnablaK v. 
			\]
			For $q, q'\in \mathbb{P}_k(K)$ arbitrary, as $\QnablaK q'=0$, we can write
			\[
			| \QnablaK(\resto v) |_{1,K} =  | \QnablaK(\resto (v-q) + \resto q - q') |_{1,K}   \lesssim | \resto (v - q) |_{1,K} + | \resto q - q' |_{1,K}.
			\]
			We now choose $q=\PinablaK v$ and $q'=\PinablaK(\rho q)$. With this choice, \blu{using Proposition \ref{prop:3.2}},  we can write
			\begin{multline*}	| \resto (v - q) |_{1,K} \lesssim \| \blu{\rho^K} \|_{0,\infty,K} | v - \PinablaK v |_{1,K}  + \| \blu{\rho^K} \|_{1,\infty,K} \| v - \PinablaK v \|_{0,K}\\[1.8mm] \lesssim 
		h_K | v - \PinablaK v |_{1,K}  +  \| v - \PinablaK v \|_{0,K}
				\lesssim h_K | v |_{1,K}, \end{multline*}
			as well as
			\begin{multline*}
			| \resto q - q' |_{1,K} = | \resto q - \PinablaK(\resto q) |_{1,K}  \lesssim h_K | \resto q |_{2,K} \lesssim h_K \| \resto \|_{0,\infty,K} | q |_{2,K} + h_K\| \resto \|_{2,\infty,K} \| q \|_{1,K}   \\[1.8mm]
			\lesssim h_K^2 | q |_{2,K} + h_K \| q \|_{1,K} \lesssim h_K \| q \|_{1,K} \lesssim h_K \| v \|_{1,K},
			\end{multline*}
			\blu{where we also used \eqref{eq:polyinverse}}. 
		\end{proof}	
	
	\

	The third commutativity property that we will need in the forthcoming analysis is stated in the lemma below.
	
	\begin{lemma}\label{lem:commDK} Let $\omega \in C^\infty(\bar\Omega)$ be a fixed weight function. Then, for all $K \in \Tess$, for all $w,v \in H^1(K)$ it holds that
		\begin{equation}\label{commDK}
	| \DK(w ,\weight v) - \DK (\weight w, v) |  \lesssim h_K \| v \|_{1,K} \| w \|_{1,K},
		\end{equation}
		the implicit constant in the inequality depending on $\| \weight \|_{2,\infty,\Omega}$.		
	\end{lemma}	

\begin{proof}	
Using once more the splitting \eqref{splittingweight}, we have
\begin{multline*}
| \DK(w ,\weight v) - \DK (\weight w, v) |  = |  \DK(w, \resto v) - \DK(\resto w , v)  | \\ \lesssim | \QnablaK w |_{1,K} |\QnablaK (\resto v) |_{1,K} + | \QnablaK (\resto w) |_{1,K} |\QnablaK v |_{1,K}.
\end{multline*}
Now, using Proposition \ref{boundQnablaK} and Lemma \ref{lem:4.4},   we have
\begin{multline*}
| \QnablaK (\resto w) |_{1,K} \leq | \resto \QnablaK w  |_{1,K} + |\QnablaK (\resto w) - \resto \QnablaK w |_{1,K} \\[1.8mm]
\lesssim \| \resto \|_{0,\infty} | \QnablaK w |_{1,K} + \| \resto \|_{1,\infty,K} \| \QnablaK w \|_{0,K} + h_K \| w \|_{1,K} \lesssim h_K \| w \|_{1,K}.
\end{multline*}
The above bound also applies to $v$, finally yielding \eqref{commDK}.
\end{proof}

		\section{Negative norm error estimates}
	The interior error estimate we aim at proving relies on the validity of different bounds on the error measured in negative norms, both at the global and at the local level. We devote this section to study such bounds. 
	
		\subsection{Error bounds in the $H^{-p}(\Omega)$ norm}\label{sec:negativeglobal}We start by considering the error $e = u - u_h$ measured in the  $H^{-p}(\Omega)$ norm.  We assume that  $u \in H^{r+1}(\O)$, $0 \leq r \leq k$  and let $\rho$, with  $0 \leq \rho \leq m+1$, be such that $f = - \Delta u \in H^\rho(\O)$. 
		As usual, resorting to a duality argument, in order to bound $\| e \|_{-p,\O}$,  $p \geq 0$, we write
		\begin{equation}
			\label{duality}
		\| e \|_{-p,\Omega} = \sup_{\phi\in H^p_0(\Omega)} \frac{\int_\Omega e \phi}{\| \phi \|_{p,\Omega}} =  \sup_{\phi\in H^p_0(\Omega)} \frac{\int_\Omega \nabla e \cdot \nabla  \vphi}{\| \phi \|_{p,\Omega}},
		\end{equation}
	where $\vphi \in H^1_0(\Omega)$ is the solution of 
		\begin{equation}\label{defvphi}
		-\Delta \vphi = \phi, \quad \text{ in }\Omega, \qquad \vphi = 0, \quad \text{ on }\partial \Omega.
		\end{equation}
Adding and subtracting $\vh \in \Vh$ arbitrary,	using \blu{\eqref{eq:conformity} and}
 \eqref{eq:error}, and adding and subtracting $\vphi$ we have that
		\begin{multline}
		\int_\Omega \nabla e \cdot \nabla  \vphi = \int_\Omega \nabla e \cdot \nabla  (\vphi - \vh) + \int_\Omega \nabla e \cdot \nabla  \vh \\
		%=  \int_\Omega \nabla e \cdot \nabla  (\vphi - \vh) + a_h( e, \vh) + \Dh(e, \vh) \\
		 = 
		\int_\Omega \nabla e \cdot \nabla  (\vphi - \vh) + \int_{\Omega} \df  \vh - \Dh(u , \vh) + \Dh(e, \vh) \\
		=
		\int_\Omega \nabla e \cdot \nabla  (\vphi - \vh) + \int_{\Omega} \df  (\vh - \vphi  ) + \int_{\Omega}\df  \vphi \\ - \Dh(u , \vh - \vphi ) + \Dh(e, \vh - \vphi) - \Dh(u ,  \vphi ) + \Dh(e,  \vphi).\label{5.1.b}
		\end{multline}
	Estimates of the right hand side of \eqref{5.1.b} will, as usual, rely on the smoothness lifting properties of the Dirichlet problem \eqref{defvphi}. 
	 In the optimal case (e.g. when $\Omega$ is a square or a smooth domain) we will have $\vphi \in H^{p+2}(\O)$ \blu{(for the lifting property on squared domains see \cite[eqn. 7.16]{NS})}. However, this will not always be the case here, 
	as, on polygonal domains, depending on the interior angles, the 
	smoothness of $\phi$ implies the smoothness of $\vphi$ only up to a certain limit. In general (see \cite{grisvard2011elliptic}), we will have  that 
%$\phi \in H^p_0(\Omega)$ implies  $\vphi \in H^{s+1}(\Omega)$ for some $s$ depending on $p$ and $\Omega$. More precisely
	\begin{equation}\label{lifting}
		\phi \in H^p_0(\Omega) \quad \text{implies} \quad
\| \vphi \|_{1+s,\Omega} \lesssim \| \phi \|_{p,\Omega}, \qquad\text{for some } s = s(p,\Omega) > 1/2.
\end{equation}
We then take $\vh = \fIh \vphi$ (as $s > 1/2$, $H^{1+s}(\Omega) \subseteq C^0(\bar\Omega)$ so that $\fIh \vphi$ is well defined). Using \eqref{interperror}, \eqref{errorDeltaK}  and \eqref{errorrhs} to bound the right hand side of \eqref{5.1.b} we obtain
		\begin{gather}\label{boundA}
		\int_\Omega \nabla e \cdot \nabla  \vphi 	\lesssim 
	 (| e |_{1,\Omega} + h^r | u |_{r+1,\Omega}) h^{\min\{s,k\}} | \vphi |_{s+1,\Omega}  + h^{\rho+\min\{s,m\} + 1} \| f \|_{\rho,\Omega} | \vphi |_{s+1,\Omega}.
		\end{gather}

		%\qquad \sk = \min\{s,k\},\qquad \sm=\min\{s,m\}.\]
		
		\NOTE{We have ($m \leq k$), with  $\sk = \min\{s,k\}$, $\sm=\min\{s,m\}$
			\begin{gather*}
				\int_\Omega \nabla e \cdot \nabla  (\vphi - \vh) + 	\int_\Omega  \Dh(e, \vh - \vphi) \leq | e |_{1,\Omega} | \vphi - \vh |_{1,\O} \lesssim  | e |_{1,\Omega} h^\sk | \vphi |_{1+s,\O}  \\[1mm]
				 \int_{\Omega} \df  (\vh - \vphi  ) \lesssim h^\rho | f |_{\rho,\Omega} h | \vh - \vphi |_{1,\Omega} \lesssim h^\rho | f |_{\rho,\Omega} h^{\sk + 1}| \vphi |_{1+s,\O} \lesssim  h^\rho | f |_{\rho,\Omega} h^{\sm + 1} | \vphi |_{1+s,\O}\\[1mm]
				 \int_{\Omega}\df  \vphi \lesssim h^\rho | f |_{\rho,\Omega} h^{\sm + 1} | \vphi |_{1+s,\O}\\[1mm]
					\int_\Omega  \Dh(u , \vh - \vphi ) \lesssim h^{r} | u |_{r+1,\O} | \vh - \vphi |_{1,\O} \lesssim h^{r} | u |_{r+1,\O} h^\sk |  \vphi |_{s+1,\O}\\[1mm]
				\int_\Omega 	 \Dh(u ,  \vphi ) \lesssim h^r | u |_{r+1,\O} h^\sk | \vphi |_{s+1,\O}\\[1mm]
				\int_\Omega 	 \Dh(e,  \vphi) \lesssim | e |_{1,\O} h^\sk | \vphi |_{s+1,\O}
			 			\end{gather*}
	and then we use the standard bound for the VEM error. Putting everything together yields
	\[
	\int \nabla e \cdot \nabla \vphi \lesssim  | e |_{1,\Omega} h^\sk | \vphi |_{1+s,\O} + h^\rho | f |_{\rho,\Omega} h^{\sm + 1} | \vphi |_{1+s,\O} + | \vphi |_{s+1,\O}
	\]
}
		Using \eqref{boundA} \blu{and \eqref{lifting}} in \eqref{duality}, and bounding $| e |_{1,\Omega}$ thanks to  \eqref{stimaerrorestandard},  we then obtain the following bound
		\begin{multline}
		\| e \|_{-p,\O} \lesssim  h^{\min\{s,k\}}\big(h^r | u |_{r+1,\O}+ | e |_{1,\O}\big) + h^{\min\{s,m\} + \rho + 1}  \blu{\| f \|_{\rho,\Omega}}  \\[1mm]
		\lesssim h^{\min\{s,k\} + r } | u |_{r+1,\O} + h^{\min\{s,m\} + \rho + 1} \blu{\| f \|_{\rho,\Omega}},
		\end{multline}
		where we used the fact that $m \leq k$.

		\
		
	It then remains to see which is the value of $s(p,\Omega)$, that is, what is the regularity that, depending on the characteristics of the domain $\Omega$, we can expect for $\vphi$ if $\phi \in H^{p}_0(\Omega)$. As already recalled, if $\Omega$ is smooth or a square, we will have that $\vphi \in H^{p+2}(\Omega)$, that is $s(p,\Omega) = p+1$. If, instead, $\Omega$ is a polygon  (\cite{grisvard2011elliptic}), we know that if $\phi \in H^{\sigma-1}(\Omega)$ then $\vphi \in H^{\sigma+1}(\Omega)$ and
	\[
	\| \vphi \|_{\sigma+1,\Omega} \lesssim \| \phi \|_{\sigma - 1,\Omega},
	\]
provided $ 0 < \sigma < \sigma_0 = \frac{\pi}{\theta_0}$, where
 $\theta_0 = \max_i \theta_i$, $\theta_i$, $i = 1,\cdots,L$ denoting the interior angles at the $L$ vertices of $\Omega$.
		Observe that we have that $\pi/3 \leq \theta_0 < 2\pi$, whence $1/2 < \sigma_0 \leq 3$.	For $\epsilon > 0$ arbitrarily small but fixed, we can then take $s=s(p,\Omega) = \min\{p+1,\sigma_0 - \varepsilon\}$ in \eqref{boundA}. 
			Setting 
		\[\g = \min\{p+1,k\} \qquad \text{ and }\qquad \s = \min \{p+1,m\}\] we then have 
		\begin{equation} \| e \|_{-p,\Omega}  \lesssim
	h^{\min\{\g,\sigma_0-\varepsilon\}+r} | u |_{r+1,\Omega}  +  h^{\rho + 1 + \min\{\s,\sigma_0-\varepsilon\} }\| f \|_{\rho,\Omega}.\label{negativepolygongeneral}
		\end{equation}
	
\NOTE{Setting $\rho = \max\{r - 1,0\}$
	Observe that for $p = k-1$ we have that $\g = k$ and  $\s = m$. Moreover, in the range of values we are interested in, we have that
\[
A := \min\{
\min\{k,\sigma_0-\varepsilon\}+r, \rho + 1 + \min\{m,\sigma_0-\varepsilon\} \}= \min\{
\sigma_0 - \varepsilon + r, m + \rho + 1
\} =: B .
\]
To see this, we start by observing that $s = \rho+1 \geq \max\{
r,1
\}$. We distinguish two cases: $r \in (0,1)$ (and then $s \geq 1$) and $r \geq 1$. With $\sigma = \sigma_0 - \varepsilon$
\begin{table}[htbp]
	\begin{tabular}{l||c|c|c|c||c|c|c|c}
		& \multicolumn{4}{c||}{$r \geq 1$}& \multicolumn{4}{c}{$r \in (0,1)$}\\
		$\sigma$ & $A$ & $B$  & $\min\{A,B\}$ & $\min\{C,D\}$ & $A$ & $B$ & $\min\{A,B\}$  & $\min\{C,D\}$ \\
		\hline
		$\sigma \leq m$ & C  & $\sigma$ +s &  C & C  & C&   $\sigma + s$ & C &C \\
		$m < \sigma \leq k$ & C& D& $\min\{C,D\}$ & $\min\{C,D\}$ & C & D& $\min\{C,D\}$ & $\min\{C,D\}$ \\
		$k < \sigma$ & $k + r$ & D & $D$ & $D$ & -- & -- & -- & --
	\end{tabular}
	\caption{
		$A = \min\{k,\sigma\} + r$, $B = \min\{
		m,\sigma
		\} + S$, $C = \sigma + r$, $D = \sigma + s$. Observe that if $\sigma > k$, $f \in L^2$ implies $u \in H^2$. Then, the case $\sigma > k$, $r \in (0,1)$ is irrelevant.}
\end{table}

}

\NOTE{
	For $p = k-1$, $r = \rho = 0$ we have 
	\[
	\| e \|_{1-k} \lesssim h^{\min\{k,\sigma_0-\varepsilon\}} | u |_{1,\Omega} + h^{\min\{m,\sigma_0-\varepsilon\} + 1} \| f \|_{0,\Omega}.
	\]	
	If $m = k-2$ the second (suboptimal) term dominates, thus the necessity of better approximation of the source term. 
}

			\subsection{Negative norm error estimates on smooth domains}\label{sec:negativesmooth} Before going on, let us consider what happens instead
			in  the case in which $\Omega$ is smooth. In such a case, for the solution of \eqref{defvphi}, we have that for all $p\geq 0$, $\phi \in H^p(\Omega)$ implies $\vphi \in H^{p+2}(\Omega)$. 
			We can take advantage of such a fact, provided that, for the discretization, we use a tessellation allowing curvilinear edges at the boundary of $\Omega$, and resort, for boundary adjacent elements, to the VEM with curved edges as introduced in \cite{VEM-curved}, following which,  we modify the definition of the space $\BoK$, which, for $K$ with $| \partial K \cap \partial \Omega | > 0$ becomes
			\[
			\BoK =  \{ v \in C^0(\bK):\ v|_e \in \mathbb{P}_k(e)\ \text{for all edge $e$ of $K$ interior to $\Omega$}, \ v|_{\bK \cap \partial\Omega} = 0 \}.
			\]

			We also modify	the interpolator $\fIh$ by requiring, for boundary adjacent elements, that $\fIh u = 0$ on $\partial K \cap \partial\Omega$. Using arguments similar to the ones used in Section \ref{sec:commutator} we can see that, also for boundary adjacent elements, for all $u \in H^1_0(\Omega)$ with $u|_K \in H^{1+t}(K)$, $t\geq 1$ it holds that
			\[
			| u - \fIh u |_{1,K} \lesssim h^{t} | u |_{1+t,K}.
			\]
			Then, letting, also for elements with a curved boundary edge, $\tPinablaK$ and $\QnablaK$ be defined by 
			\eqref{deftPinablaK} and \eqref{defQnablaK},
			we have, for $u \in H^1_0(\Omega)$ with $u|_K \in  H^{1+t}(K)$,
			\[
			| u - \tPinablaK u |_{1,K} \leq | u - \fIh u |_{1,K} \lesssim h^t | u |_{1+t,K}.
			\]
			Moreover (see \cite{VEM-curved}), under the same assumptions on $u$,  we have that 
			\begin{equation}\label{errorPinabla}
			| u - \PinablaK u |_{1,K}  \lesssim h^t | u |_{1+t,K}.
			\end{equation}
			Both Proposition \ref{boundQnablaK} and Proposition \ref{boundfurbo} also hold. Indeed, if $u \in H^{1+s}(K)$, $1\leq s \leq k$, with $u = 0$ on $\bK \cap \partial\Omega$ we have 
			\begin{gather*}
			\|  u - \tPinablaK u \|_{1,K} \leq 	| u - \tPinablaK u |_{1,K}  \lesssim | u |_{1,K}, \\	\| u - \tPinablaK u \|_{1,K} \leq 	| u - \tPinablaK u |_{1,K}  \lesssim h^s | u |_{1+s,K}, 
			\end{gather*}
			where we used a Poincar\'e inequality.
			By interpolation we have that \[
			\| u - \tPinablaK u \|_{1,K} \lesssim  h^s | u |_{1+s,K},  \qquad \forall s \text{ with } \blu{0 \leq s \leq k.}
			\]
			Proposition \ref{boundQnablaK} follows by a \blu{triangle } inequality. As its proof essentially relies on Proposition \ref{boundQnablaK}, Proposition \ref{boundfurbo} also follows.
			
			\
			
			Then, we can bound the right hand side of \eqref{5.1.b} as follows.
			As for $\Omega$ smooth $f \in L^2(\Omega)$ implies $u \in H^2(\Omega)$, we can always assume that $r \geq 1$ and that $\rho = r-1$. Moreover we can take $s = p+1$, so that we have
				\begin{equation}\label{negativesmoothbest}
				\| e \|_{-p,\O} 
				\lesssim \big(h^{\min\{p+1,k\} + r } + h^{\min\{p+1,m\} + r}\big) | u |_{r+1,\O}  \lesssim h^{\min\{p+1,m\} + r}| u |_{r+1,\O}.
				\end{equation}
			
			\begin{remark}
				We recall that if $\Omega$ is a square, the same smoothness result holds for the solution $\vphi$ of problem \eqref{defvphi}, as in the case of smooth domains, namely, if $\phi \in H^p_0(\Omega)$ it holds that $\vphi \in H^{p+2}(\Omega)$ and
				$\| \vphi \|_{p+2,\Omega} \lesssim \| \phi \|_{p,\Omega}$.
				Then, also in such a case we have that \eqref{negativesmoothbest} holds.
			\end{remark}

\input{localnegative}

	\section{Interior error estimate}\label{sec:main}
	
	We can now prove the main result of this paper, stating that the local error in $H^1(\G_0)$, $\G_0  \subsubset \Omega$ is bounded by a term of the maximum order allowed by the  smoothness of $u$ in $\Guno$, with $\G_0 \subsubset \Guno \subsubset \Omega$, plus the global error measured in a weaker negative norm.

\begin{theorem}\label{thm:main}
	Let $\G_0 \subsubset \Guno \subsubset \Omega$ and let $u$ and $u_h$ denote %, respectively, 
	the solution to \eqref{Pb:weak} and \eqref{pb:discrete}, respectively. Assume that $u|_\Guno \in H^{1+t}(\Guno)$, $1 \leq t \leq k$. Then, for $p \geq 0$ arbitrary, there exists $h_0$ such that, provided $h<h_0$, it holds that
	\begin{equation}\label{boundcor8b}
	\| e \|_{1,\G_0} \lesssim h^t( | u |_{t+1,\Guno} +  \| u \|_{1,\Omega} + h \| f \|_{0,\Omega}) + \| e \|_{-p,\Omega}.
	\end{equation}
\end{theorem}

In order to prove Theorem \ref{thm:main}, we start by proving the following lemma.
	
\begin{lemma}\label{lem:8} 	
	Let $\G_0 \subsubset \Guno \subsubset \Omega$ and
	let $e = u - u_h$ with $u$ and $u_h$ % respectively
	solution to 
	\eqref{Pb:weak} and \eqref{pb:discrete}, respectively. 
Assume that $u|_{\Guno} \in H^{1+t}(\Guno)$, $1 \leq t \leq k$. 
Then, there exists $h_0$ such that, provided $h < h_0$
			\[
			\| e \|_{1,\G_0} \lesssim h \| e \|_{1,\Guno}  + h^t | u |_{1+t,\Guno}  + \| e \|_{0,\Guno}.\]
\end{lemma}
	
	\begin{proof}	Let $\G'$, with $\G_0 \subsubset \G' \subsubset \Guno$, be an intermediate subdomain between $\G_0$ and $\Guno$.	Again, we let $h_0$ be such that for all $h < h_0$, all elements $K \in \Tess$ with $K \cap \G' \not= \emptyset$ satisfy $K \subset \overline{\G}_1$.
		%, and we assume $h < h_0$.			
		 Let now \[\Vho(\Guno) = \{ \vh \in \Vh: \supp \vh \subseteq\overline{\G}_1\}, \] 		
	and, letting $\omega \in C^\infty_0(\G')$, with $\weight = 1$ in $\G_0$, we let $\te = \weight e$ and $\ce \in \Vho(\Guno)$ denote the solution to 
		\[a_h(\ce,\vh) = a_h(\te,\vh), \quad \forall \vh \in \Vho(\Guno).\]
		
		It holds that
		\[
		\| e \|_{1,\G_0} \leq \| \te \|_{1,\Guno} \leq\| \te - \ce \|_{1,\Guno} + \| \ce \|_{1,\Guno}. 
		\]
	Observing that, as $h < h_0$, $\fIh (\te) \in \Vho(\Guno)$,  using \eqref{coerch} we can write
		\begin{multline*}
		\| \te - \ce \|^2_{1,\Guno} \lesssim \int_{\Tess}| \nabla(\Id-\tPinabla) (\te - \ce) |^2 + a_h(\te - \ce, \te - \ce) \\=
		\int_{\Tess} \nabla(\Id-\tPinabla) (\te - \ce)\cdot \nabla(\Id-\tPinabla)  \te + a_h(\te - \ce, \te - \fIh(\te)) \\
		\lesssim | (\Id-\tPinabla) (\te - \ce) |_{1,\Tess}\, | (\Id-\tPinabla) \te |_{1,\Tess} + | \te - \ce |_{1,\Guno} |\te - \fIh(\te) |_{1,\Guno} \\[1.8mm] \lesssim \left( |(\Id-\tPinabla) \te\, |_{1,\Tess} + |\te - \fIh(\te) |_{1,\Guno} \right)  | \te - \ce |_{1,\Guno},
		\end{multline*}
		yielding
		\[
		\| \te - \ce \|_{1,\Guno} \lesssim | (\Id-\tPinabla) \te\, |_{1,\Tess} + |\te - \fIh(\te) |_{1,\Guno} \lesssim h \| e \|_{1,\Guno} + h^t | u |_{1+t,\Guno},
		\]
		where, for the last bound, we used Lemma \ref{lem:disccom} and Corollary \ref{cor:disccom} with $v = u$ and $\vh = -\uh$.
		Let us now bound $\ce$. It holds that
		\begin{equation}\label{cebound1}
		\| \ce \|^2_{1,\Guno} \lesssim a_h (\ce,\ce) = a_h ( \te,\ce) = a_h( e, \omega \ce) + \Komega(e,\ce)
		\end{equation}
		with 
		\[
		\Komega(w,v) = a_h(\weight w, v)- a_h(w, \weight v).
		\]
Since we have that
\begin{gather*}
(\Id-\tPinabla) \ce = 0, \qquad
(\Id-\tPinabla)  e  = (\Id-\tPinabla)  u, 
\end{gather*}
we can write
\begin{multline*}
\Komega (e,\ce) =  \int_{\Tess} \nabla \Pinabla (\weight e) \cdot \nabla\Pinabla \ce + s_h(\Qnabla(\weight e),\Qnabla \ce) -  \int_{\Tess} \nabla \Pinabla e \cdot \nabla\Pinabla (\weight\ce) - s_h(\Qnabla e,\Qnabla (\weight\ce)  ) \\
= 
\int_{\Omega} \nabla(\weight e) \cdot \nabla \ce - \int_{\Tess} \nabla \Qnabla (\weight e) \cdot\nabla  \Qnabla \ce  \\ - \int_{\Omega} \nabla e \cdot \nabla (\weight\ce) + \int_{\Tess} \nabla \Qnabla e \cdot \nabla\Qnabla (\weight\ce)  + \int_{\Tess}\nabla ((\Id-\tPinabla) u) \cdot \nabla (\Id-\tPinabla) (\weight\ce) \\ + s_h(\Qnabla(\weight e),\Qnabla \ce) - s_h(\Qnabla e, \Qnabla(\weight \ce))\\
= \int_{\Omega} \nabla(\weight e) \cdot \nabla \ce - \int_{\Omega} \nabla e \cdot \nabla (\weight\ce) - \Dh(\weight e , \ce) + \Dh(e , \weight \ce) + \int_{\Tess}\nabla ((\Id-\tPinabla) u) \cdot \nabla (\Id-\tPinabla) (\weight\ce).
\end{multline*}
We  recall (see \cite{NS}) that we have (the implicit constant depending on $\weight$)
\begin{multline*}
\int_{\Omega} \nabla(\weight e) \cdot \nabla \ce - \int_{\Omega} \nabla e \cdot \nabla (\weight\ce) 
= \int_\Omega e [
\nabla \omega\cdot \nabla \ce + \nabla \cdot \ce \nabla \weight
] \\ 
\lesssim \| e \|_{0,\Guno} \|   \nabla \omega\cdot \nabla \ce + \nabla \cdot \ce \nabla \weight \|_{0,\Guno} \lesssim \| e \|_{0,\Guno} | \ce |_{1,\Guno}.
\end{multline*}
As $\weight$ is supported in $\G'$, using Lemma \ref{lem:commDK} we also have 
\begin{multline*}
 \Dh(e , \weight \ce) -  \Dh(\weight e ,\ce) = \sum_{K\in \Tess: K \subset \Guno} \left(\DK(e , \weight \ce) - \DK(\weight e ,\ce)\right) \\
 \lesssim  \sum_{K\in \Tess: K \subset \Guno}  h_K \| e \|_{1,K} \| \ce \|_{1,K} \leq h \| e \|_{1,\Guno} \| \ce \|_{1,\Guno}.
\end{multline*}
Moreover, also since $\weight$ is supported in $\G'$, we can write
\[
\int_{\Tess}\nabla ((\Id-\tPinabla) u) \cdot \nabla (\Id-\tPinabla) (\weight\ce) \lesssim | u - \tPinabla u |_{1,\G'} | \weight \ce |_{1,\Guno}
\lesssim h^t | u |_{t+1,\Guno} \| \ce \|_{1,\Guno},
\]
finally yielding
\begin{equation}\label{boundKomega}
\Komega (e,\ce) \lesssim \left(\| e \|_{0,\Guno} +  h \| e \|_{1,\Guno}  + h^t | u |_{t+1,\Guno}\right) \| \ce \|_{1,\Guno}.
\end{equation}

%===================================

We observe that, under the conditions that we assume to hold for $t$, $k$ and $m$, we have that $\min\{ t-1,m+1\} = t-1$. Adding and subtracting $\fIh(\weight \ce) $ and using \eqref{eq:error}, \eqref{disccom1}, \eqref{errorrhs} and \eqref{stabilityIK}, as $u\in H^{t+1}(\Omega_1)$ implies that $f = -\Delta u \in H^{t-1}(\Omega_1)$,  we bound
		\begin{multline}\label{5.3}
		a_h( e, \omega \ce)  = a_h( e, \weight \ce - \fIh(\weight \ce) ) + \Dh(u,\fIh(\weight \ce) ) - \int_{\Omega} \df \fIh(\weight \ce) \\[1.8mm] \lesssim 
		\| e \|_{1,\Guno} \| \weight  \ce - \fIh(\weight \ce)  \|_{1,\Guno} + \Dh(u,\fIh(\weight \ce) ) 
		\blu{-  \int_{\Omega} \df \fIh(\weight \ce)}\\[1.8mm]
		\lesssim h \| e \|_{1,\Guno} \| \ce \|_{1,\Guno} + \blu{ h^t | u |_{t+1,\Guno}} | \ce |_{1,\Guno}.
		\end{multline}
	Plugging \eqref{boundKomega} and \eqref{5.3} in \eqref{cebound1}  we obtain
		\[
		\| \ce \|_{1,\Guno}^2 \lesssim \left(h \| e \|_{1,\Guno}  + h^t | u |_{t+1,\Guno} + \| e \|_{0,\Guno}\right) \| \ce \|_{1,\Guno},\]
		from which, dividing by $\| \ce \|_{1,\Guno}$ we obtain
		\[
		\| \ce \|_{1,\Guno} \lesssim h \| e \|_{1,\Guno}  + h^t | u |_{t+1,\Guno}  + \| e \|_{0,\Guno},
		\]
whence, by \blu{triangle}  inequality
		\[
		\| e \|_{1,\G_0} \lesssim h \| e \|_{1,\Guno}  + h^t | u |_{t+1,\Guno}  + \| e \|_{0,\Guno}.\]
\end{proof}

\blu{We can now combine Lemma \ref{lem:8} with Lemma \ref{lem:7b}, and we obtain the following corollary, where $h_0 = \min\{h_0',h_0''\}$, $h_0'$ given by Lemma \ref{lem:8} on $ \G_0'\subsubset \G_1'$ and $h_0''$ given by Lemma \ref{lem:7b} on $ \G_0''\subsubset \G_1''$, where $\G_0 = \G_0'$, $\G_1 = \G_1''$, and where $\G_1'= \G_0''$  denote an intermediate subdomain. 
\begin{corollary}\label{cor:8}
	Under the assumptions of Lemma \ref{lem:8}, for $p > 0$ arbitrary, there exists $h_0$ such that, provided $h < h_0$
	\[
	\| e \|_{1,\G_0} \lesssim h \| e \|_{1,\Guno}  + h^t | u |_{1+t,\Guno}  + \| e \|_{-p,\Guno}.\]
\end{corollary}
}

We can now prove Theorem \ref{thm:main}.

\begin{proof}[Proof of Theorem \ref{thm:main}] Let $q$, $p$ be  arbitrary positive integers and 
		let once again $\hOmega_\ell$, $\ell = 0,\cdots,q$ be intermediate subdomains with
		\(	\G_0= \hOmega_0 \subsubset \hOmega_1 \subsubset \cdots \subsubset \hOmega_q = \Guno\). By Corollary \ref{cor:8}, for $\ell  = 0,\cdots,q$, there exists $h_{0,\ell}$ such that, provided $h < h_{0,\ell}$, the bound 
			\[
			\| e \|_{1,\hOmega_{\ell}} \lesssim h \| e \|_{1,\hOmega_{\ell+1}}  + h^t | u |_{1+t,\hOmega_{\ell+1}}  + \blu{\| e \|_{-p,\hOmega_{\ell+1}}}\]
holds. Then, \blu{if $h < h_0 = \min_{\ell}\{ h_{0,\ell} \}$} we can write
\begin{multline}
\| e \|_{1,\G_0} \lesssim h \| e \|_{1,\hOmega_{1}}  + h^t | u |_{1+t,\hOmega_{1}}  + \blu{ \| e \|_{-p,\hOmega_{1}}}\\[2mm]
 \lesssim
h^2  \| e \|_{1,\hOmega_{2}}  +  (1 + h) h^t | u |_{1+t,\hOmega_{2}}  + (1+h) \blu{ \| e \|_{-p,\hOmega_{2}}}
 \lesssim \cdots \\
\lesssim 
h^q \| e \|_{1,\hOmega_q} + \left(\sum_{\ell=0}^{q-1} h^{\ell}\right) h^t | u |_{1+t,\hOmega_{q}} +  \left(\sum_{\ell=0}^{q-1} h^{\ell}\right)\blu{ \| e \|_{-p,\hOmega_{q}}}.
\end{multline}
As $\sum_{\ell=0}^{q-1} h^\ell \lesssim 1$, this yields
%\begin{equation}
%\| e \|_{1,\G_0} \lesssim h^q \| e \|_{1,\blu{\Omega_1}} +  h^t | u |_{1+t,\blu{\Omega_1}}} + \blu{ \| e \|_{-p,\hOmega_{1}}}.
%\
%By Lemma \ref{lem:7b}, there exists $h_2$ such that, if $h < h_2$, we have 
%\[
%\| e \|_{0,\hOmega_q} \lesssim h | e |_{1,\Guno} + h^t | u |_{t+1,\Guno} + \| e \|_{-p,\Guno}.
%\] 
%If $h < h_0 = \min \{ h_1,h_2\}$ then 	
\begin{equation}
\| e \|_{1,\G_0} \lesssim h^q \| e \|_{1,\Guno} + h^t | u |_{t+1,\Guno}  + \| e \|_{-p,\Guno}.\end{equation} 
Choosing $q \geq t$, and using \eqref{stimaerrorestandard} for $r = \rho = 0$ yields 
\[
\| e \|_{1,\G_0} \lesssim h^{t} (\| e \|_{1,\Omega} + | u |_{t+1,\Guno}) + \| e \|_{-p,\Omega}
 \lesssim h^{t} (| u |_{1,\Omega} + h \| f \|_{0,\Omega} + | u |_{t+1,\Guno}) + \| e \|_{-p,\Omega}.
\]
		\end{proof}
		
		\NOTE{
			Intermediate bounds
			\[
			\| e \|_{1,\G_0} \lesssim h^{t} (\| e \|_{1,\Guno} + | u |_{t+1,\Guno}) + \| e \|_{-p,\Guno}.
			\]
			We recall \eqref{stimaerrorestandard}
			\begin{equation*}
			\| u - \uh \|_{1,\Omega} \lesssim h^{\min\{r,k\}} | u |_{1+r,\Omega} + h^{\min\{\rho,m+1\} + 1} | f |_{\rho,\Omega}.
			\end{equation*}
			and the global negative norm estimates. For polygonal domains and $p = k-1$ we have				\[ \| e \|_{-p,\Omega}  \lesssim
				h^{\min\{k,\sigma_0-\varepsilon\}+r} | u |_{r+1,\Omega}  +  h^{\rho + 1 + \min\{m,\sigma_0-\varepsilon\} }\| f \|_{\rho,\Omega}.\label{negativepolygongeneral}
	\]
In the minimal smoothness case	($r = \rho = 0$) this reduces to:
			\[
			\| e \|_{1-k} \lesssim h^{\min\{k,\sigma_0-\varepsilon\}} | u |_{1,\Omega} + h^{\min\{m,\sigma_0-\varepsilon\} + 1} \| f \|_{0,\Omega}.
			\]	
			If $m = k-2$ the second (suboptimal) term dominates, thus the necessity of better approximation of the source term. For smooth domain we have
				\begin{equation*}
				\| e \|_{-p,\O} 
				\lesssim  h^{\min\{p+1,m\} + r}| u |_{r+1,\O}.
				\end{equation*}
				
		}
		
		\newcommand{\effectiveorder}{\kappa}
In order to obtain an explicit a priori estimate on the local error we finally  combine \eqref{boundcor8b} with the global negative norm error estimates of Sections \ref{sec:negativeglobal} and \ref{sec:negativesmooth}. We distinguish two cases: $\Omega$ polygon and $\Omega$ smooth.  If $\Omega$ is a polygon, we can use the bound \eqref{negativepolygongeneral}, and,  in \eqref{boundcor8b}, choose \blu{$p = k-1$}.  We immediately obtain the following corollary,  where $\sigma_0$ is the domain dependent parameter defined in Section \ref{sec:negativeglobal}, related to the interior angles of  $\Omega$.
\begin{corollary}\label{cor:finalpoly} 	Let $\G_0 \subsubset \Guno \subsubset \Omega$, with $\Omega$ polygonal domain, and let $u$ and $u_h$ denote, respectively, the solution to \eqref{Pb:weak} and \eqref{pb:discrete}. 
Assume that \blu{$f \in H^\rho(\Omega)$} and  $u \in H^{1+r}(\Omega)$ for some \blu{$\rho$} and $r$ with $0\leq r\ \blu{\leq k}$ and \blu{$\max\{0,r-1\} \leq \rho \leq m+1$}. Assume also that $u|_\Guno \in H^{1+t}(\Guno)$, with $\max\{1,r\} \leq t \leq k$. Then
\[	\| u - u_h \|_{1,\G_0} \lesssim h^\effectiveorder, \quad \text{ with } \effectiveorder = \min\{t,m+\rho+1,\sigma_0-\varepsilon+r\}.
\]
\end{corollary}
 For $r = \rho = 0$ we obtain the following bound, valid under the minimal global regularity assumptions on $u$, namely $u \in H^1(\Omega)$, $f = -\Delta u \in L^2(\Omega)$:
 \[
	\| u - u_h \|_{1,\G_0} \lesssim h^{\min\{t,\sigma_0-\varepsilon\}} (| u |_{1+t,\Guno} + \| u \|_{1,\Omega})+ h^{\min\{t,m,\sigma_0-\varepsilon\} + 1} \| f \|_{0,\Omega}.
	\]

If, on the other hand, $\Omega$ is smooth (or if $\Omega$ is a square) using once again \eqref{boundcor8b} with  $p = k-1$ yields the following bound:
\[\| e \|_{1,\G_0} \lesssim h^t | u |_{1+t,\Guno} + \big(h^t + h^{k + r} \big) \| u \|_{1+r,\Omega} + \big(
h^{t+1} + h^{m + \rho + 1} 
\big) \| f \|_{\rho,\Omega}.
\]
This time, we have the following corollary.
\begin{corollary}\label{cor:finalsmooth}
	Let $\G_0 \subsubset \Guno \subsubset \Omega$, with $\Omega$ smooth domain, and let $u$ and $u_h$ denote, respectively, the solution to \eqref{Pb:weak}, and the solution to \eqref{pb:discrete} obtained with the discretization considered in Section \ref{sec:negativesmooth}. Assume that $u \in H^{1+r}(\Omega)$ for some $r$ with \blu{$1 \leq  r \leq k$} and that $u|_\Guno \in H^{1+t}(\Guno)$, with \blu{$r \leq t \leq k$}. Then we have
\[	\| u - u_h \|_{1,\G_0} \lesssim h^\effectiveorder, \quad \text{ with } \effectiveorder = \blu{\min\{t,m+r\}}.
\]
\end{corollary}

Under the minimal  global regularity assumption on $f$, namely  $f = -\Delta u \in L^2(\Omega)$, this time we have that
	\[
	\| u - u_h \|_{1,\G_0} \lesssim h^t (| u |_{1+t,\Guno} + \| u \|_{1,\Omega})+ h^{\min\{t,m\} + 1} \| f \|_{0,\Omega}.\]
Observe that we do not have optimality unless $m \geq k-1$.

\begin{remark}
	While, for the sake of simplicity, we focused our analysis on homogeneous Dirichlet boundary conditions, the result presented in Section \ref{sec:main} extends also to other boundary conditions, such as non homogeneous Dirichlet, or mixed. In such cases, depending on the smoothness of the boundary data, the solution $u$ might lack overall regularity also for very smooth right hand side $f$.
	\end{remark}

\newcommand{\VEMKen}{V^k_\text{en}(K)}
\newcommand{\Vhen}{V^\text{en}_h}
\newcommand{\huh}{\widehat u_h}
\newcommand{\hvh}{\widehat v_h}
\newcommand{\hwh}{\widehat w_h}

\blu{

\subsection{Extension to the enhanced virtual element method}\label{sec:enhanced} 
	Let us briefly sketch how the interior estimate \eqref{boundcor8b} can be extended to a particularly relevant version of the virtual element method, namely the  {\em enhanced} VEM \rosso{\cite{EnhancedVEM}}, characterized by a local discretization space  defined as
\[
\VEMKen = \{ v \in H^1(K): \ v|_{\bK} \in \BoK, \ \Delta v \in \mathbb{P}_{k}(K),\text{ and }
(v - \PinablaK v) \perp (\Poly{k}(K)\cap\Poly{k-2}(K)^\perp)
\},
\]
where the orthogonality is intended with respect to the $L^2(K)$ scalar product, and where $\Poly{k}(K)\cap\Poly{k-2}(K)^\perp$ denotes the space of those polynomials of degree at most $k$ that are orthogonal, in $L^2(K)$, to all polynomials of degree at most $k-2$. 
Letting $\Vhen \subset H^1_0(\O)$ denote the corresponding global virtual element space, we consider the problem: find $u_h \in \Vhen$ such that for all $v_h \in \Vhen$
\begin{equation}\label{prob:en}
a_h(u_h,v_h) = \int_{\O} \Pi^0_k f v_h,
\end{equation}
where $a_h$ is defined as before.
The space $\Vhen$ does not fall into the framework which we considered up to now, since for no value of $m$ we have that $\VEMKen = \VEMK$. As a consequence, the proof of Lemma \ref{lem:disccom} is not valid for such a space. However we know (see \cite{EnhancedVEM}) that the functions in $\VEMKen$ and in $V^k_{k-2}(K)$ have the same set of degrees of freedom, and that, letting $v_h, w_h \in \VEMKen$ and $\widehat v_h, \widehat w_h \in V^k_{k-2}(K)$ denote two couples of functions 
satisfying
\begin{equation}\label{enhancedequiv1}
\hvh= v_h,\quad \hwh = w_h\ \text{ on }\bK, \qquad \int_K \hvh q = \int_K v_h q, \quad\int_K \hwh q = \int_K w_h q, \ \forall q \in \Poly{k-2}(K),
\end{equation}
(which is equivalent to saying that the value of all the  degrees of freedom of $v_h$, $w_h$ coincide, respectively, with those of $\hvh$ and $\hwh$) we have  
\begin{equation}\label{enhancedequiv2}
 \PinablaK v_h = \PinablaK \hvh, \qquad \text{ and }\qquad a_h (v_h,w_h) = a_h(\hvh,\hwh).
\end{equation}
It is then not difficult to check that  $u_h \in V_h^\text{en}$ is the solution of \eqref{prob:en} if and only if the corresponding function $\huh \in V_h$ ($V_h$ being the ``plain'' VEM space defined in \eqref{defVh} with $m = k-2$)  is solution of the modified problem: find $\huh \in V_h$ such that for all $\hvh \in V_h$
\begin{equation}\label{prob:enequiv}
a_h(\huh,\hvh) = \int_\O f \Pi^0_\text{en} \hvh,
\end{equation}
where the enhanced projection $\Pi^0_\text{en}: H^1(\Tess) \to \Poly{k}(\Tess)$ is defined, element by element, as $\Pi^0_\text{en} \hvh|_K \in \Poly{k}(K)$ such that
\[
\int_K (\Pi^0_\text{en} \hvh - \hvh) q = 0, \ \forall q \in \Poly{k-2}(K) \ \text{ and }\ \int_K (\Pi^0_\text{en} \hvh - \PinablaK \hvh) q = 0, \ \forall q \in \Poly{k}(K) \cap\Poly{k-2}(K)^\perp.
\]
Apart from the definition of the right hand side, equation \eqref{prob:enequiv} falls in the framework studied in the previous sections. It is not difficult to check that for $f \in H^r(K)$, $v\in H^{t+1}(K)$, $0\leq r \leq k-1$, $0 \leq t \leq k$ we have
\[
\int_K \delta_f v = \int_K f (v - \Pi^0_\text{en} v) = \int_K (f - \Pi^0_{k-2} f)(v - \Pi^0_\text{en}v) \lesssim h^{r+t+1}| f |_{r,K} | v |_{t+1,K}.
\]
Thanks to this inequality, used for those bounds affected by the altered right hand side, particularly  Proposition \ref{boundrhs} and Lemma \ref{lem:6b}, our analysis  carries over, with minor modifications, to Problem \eqref{prob:enequiv}. 
Moreover, thanks to the higher approximation order of $\Pi^0_\text{en}$ with respect to $\Pi^0_{k-2}$ ($\Pi^0_\text{en}$ satisfies an error bound similar to the one in Proposition \ref{prop:3.2}, \cite{EnhancedVEM}) Lemma \ref{lem:6b} now holds with $\tau = \min\{p+1,k\}$. Then Theorem \ref{thm:main}, as well as its corollaries, hold, and provide optimal local error bounds for $u - \huh$. More precisely, letting  $\Omega' \subsubset \Omega_1 \subsubset \Omega$, for $h$ small enough, under the minimal global regularity assumptions on $u$ and $f$, we have that $u \in H^{t+1}(\Guno)$, with $t \leq k$, implies
\begin{equation}\label{enhancedmain}
\| u - \huh \|_{1,\G'} \lesssim h^{\kappa} (| u |_{1+t,\Guno}  + \| f \|_{0,\Omega}), \qquad k = \begin{cases}
	\min\{t,\sigma_0-\varepsilon\} & \text{if $\Omega$ is a polygon,}\\
	t & \text{if $\Omega$ is smooth.}
	\end{cases}
\end{equation}

\

Let now $\Omega_0 \subsubset \Omega'$ and assume  that $h$ is sufficiently small so that  $\Omega^+_0 \subset \G'$, where  $\Omega_0^+ = \cup_{K\in \mathcal{T}_h^+} K$ with $\mathcal{T}_h^+ = \{
K \in \Tess : K \cap \G_0 \not=\emptyset\}$ being the set of all elements that have non empty intersection with $\G_0$.  To bound $u - u_h$ in $\G _0$ we start by observing that, by triangle inequality and \eqref{errorPinabla} we have
\[
\| u - u_h \|_{1,\Omega_0} \leq \| u - \Pinabla u \|_{1,\Omega_0} + \| \Pinabla u - u_h \|_{1,\Omega_0} \lesssim h^t | u |_{t+1,\Omega_0^+} + \| \Pinabla u - u_h \|_{1,\Omega_0^+}.
\] 
To bound the second term on the right hand side, we add and subtract, element by element, the boundary average, apply a triangle inequality  and use a Poincar\'e inequality to bound the $L^2$ norm of the boundary-average free  terms with their $H^1$ seminorm, which, in turn, is bound 
 using \eqref{3.5b}, thus obtaining
\begin{multline}\label{enhancedtransfer}
 \| \Pinabla u - u_h \|^2_{1,\Omega_h}  \lesssim  \sum_{{K \in \Tess^+}}\left(
 \Big| \int_{\bK} (\Pinabla u - u_h) \Big|^2 + 
 a^K_h(
 \Pinabla u - u_h,\Pinabla u - u_h
 ) \right)\\= \sum_{{K \in \Tess^+}}\left(
\Big| \int_{\bK} (\Pinabla u - \huh) \Big|^2 + 
 a^K_h(
 \Pinabla u - \huh,\Pinabla u - \huh
 ) \right) \\
 \lesssim \| \Pinabla u - \huh \|^2_{1,\Omega'} \lesssim \| u - \Pinabla u \|^2_{1,\Omega'} + \|  u - \huh \|^2_{1,\Omega'},
\end{multline}
where we could replace $u_h$ with $\huh$ thanks to \eqref{enhancedequiv1} and \eqref{enhancedequiv2} and where we used \eqref{3.5b} once again. 
\NOTE{$ \bar v := \fint_{\bK} (\Pinabla u - u_hh) = | \bK |^{-1} \int_{\bK} (\Pinabla u - u_h)$. We have
\begin{multline*}
\| \Pinabla u - u_h \|^2_{0,K} + | \Pinabla u - u_h |_{1,K}^2 \lesssim \| \bar v \|_{0,K}^2 + \| ( \Pinabla u - u_h ) - \bar v \|_{0,K}^2  + | \Pinabla u - u_h |_{1,K}^2 \\
\lesssim | K | |\bK |^{-2} | \int_{\bK} ( \Pinabla u - u_h ) |^2 + (1+h^2) | \Pinabla u - u_h |_{0,K} \\
\lesssim | \int_{\bK} ( \Pinabla u - u_h ) |^2 +  | \Pinabla u - u_h |_{0,K} \lesssim | \int_{\bK} ( \Pinabla u - u_h ) |^2 + a_h(\Pinabla u - u_h ,\Pinabla u - u_h ).
\end{multline*}
We also have
\begin{multline*}
| \int_{\bK} (\Pinabla u - \huh) |^2 + a_h(\Pinabla u - \huh ,\Pinabla u - \huh ) \lesssim | \bK | \int_{\bK} |\Pinabla u - \huh|^2 + | \Pinabla u - \huh |_{1,K}^2\\ = | \bK | \| \Pinabla u - \huh \|_{0,\bK}^2 + | \Pinabla u - \huh |_{1,K}^2 
\lesssim | \bK | (h^{-1}\| \Pinabla u - \huh \|^2_{0,K} + h | \Pinabla u - \huh |^2_{1,K}) + | \Pinabla u - \huh |_{1,K}^2 \\ \lesssim \| \Pinabla u - \huh \|_{0,K}^2 + (1+h^2) | \Pinabla u - \huh |_{1,K}^2
\end{multline*}
}
Combining \eqref{enhancedtransfer} with \eqref{errorPinabla} and  \eqref{enhancedmain}  finally yields the optimal error bound
\begin{equation}\label{enhancedmain}
	\| u - u_h \|_{1,\G_0} \lesssim h^{\kappa} (| u |_{1+t,\Guno}  + \| f \|_{0,\Omega}), \qquad k = \begin{cases}
		\min\{t,\sigma_0-\varepsilon\} & \text{if $\Omega$ is a polygon,}\\
		t & \text{if $\Omega$ is smooth.}
	\end{cases}
\end{equation}

}

%\COMMENT{
%	Moreover we have
%	\[
%	\| \nabla v^H \cdot n_K \|_{-1/2,K} = \sup_{\phi \in H^{1/2}(\bK)} \frac{\int_{\bK} \nabla v^H \cdot n_K \phi}{\| \phi \|_{1/2,K}} =  \sup_{\phi \in H^{1/2}(\bK)} \frac{\int_{\bK} \nabla v^H \cdot \nabla \mathcal{L}_K \phi}{\| \phi \|_{1/2,K}} \leq | \nabla v^H |_{1,K} 
%	\] 
%	where $\mathcal{L}_K: H^{1/2}(\bK) \to H^1(K)$ is a bounded lifting operator for which the constant does not depend on $K$ (the push forward of the harmonic lifting in the ``reference'' domain ). 
%}
%

\begin{figure}
	\centering
	{\includegraphics[width=0.4\textwidth]{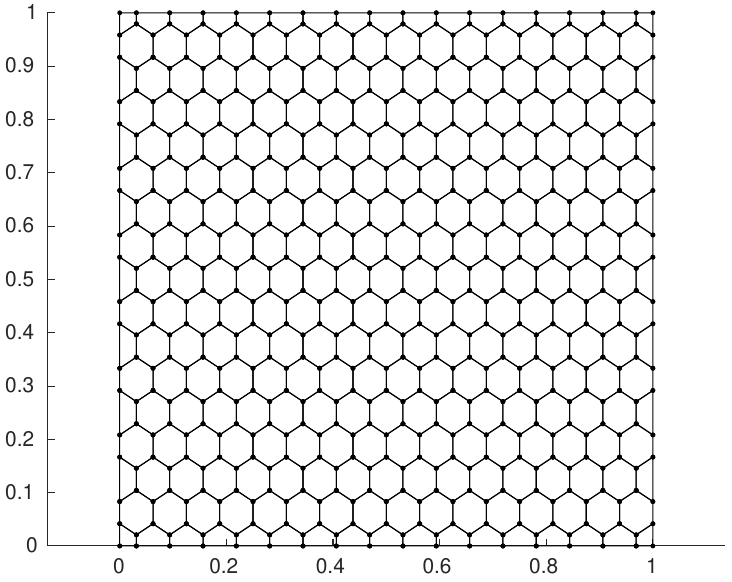}}\quad
	{\includegraphics[width=0.4\textwidth]{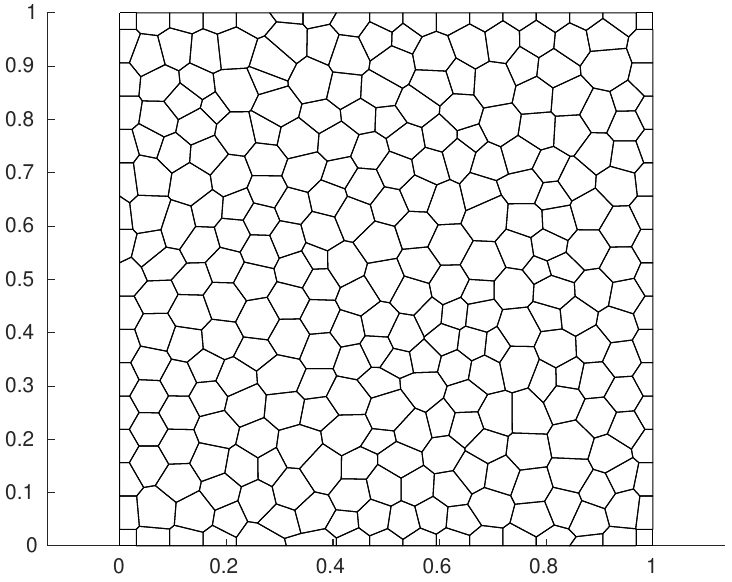}}
	\caption{Examples meshes of the unit square.}
	\label{fig:square-meshes}
\end{figure}

\begin{figure}
	\centering
	{\includegraphics[width=0.4\textwidth]{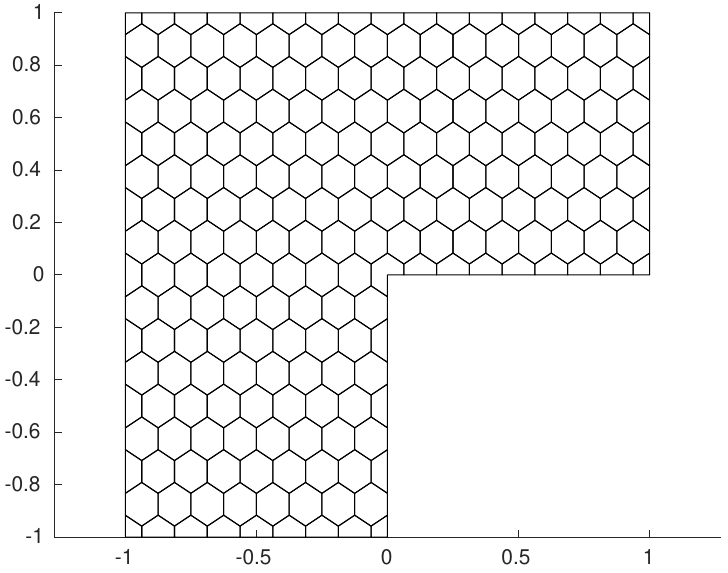}}\quad
	{\includegraphics[width=0.4\textwidth]{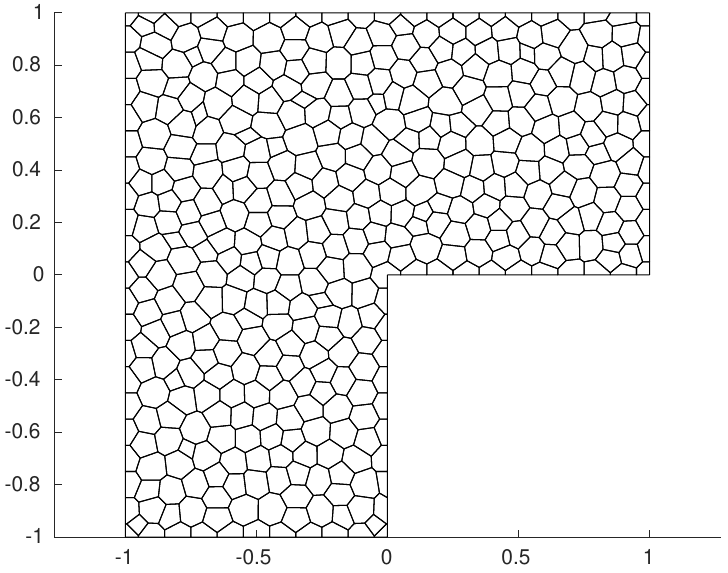}}
	\caption{Examples meshes of the L-shaped domain.}
	\label{fig:lshape-meshes}
\end{figure}

\begin{figure}
	\centering
	{\includegraphics[width=0.45\textwidth]{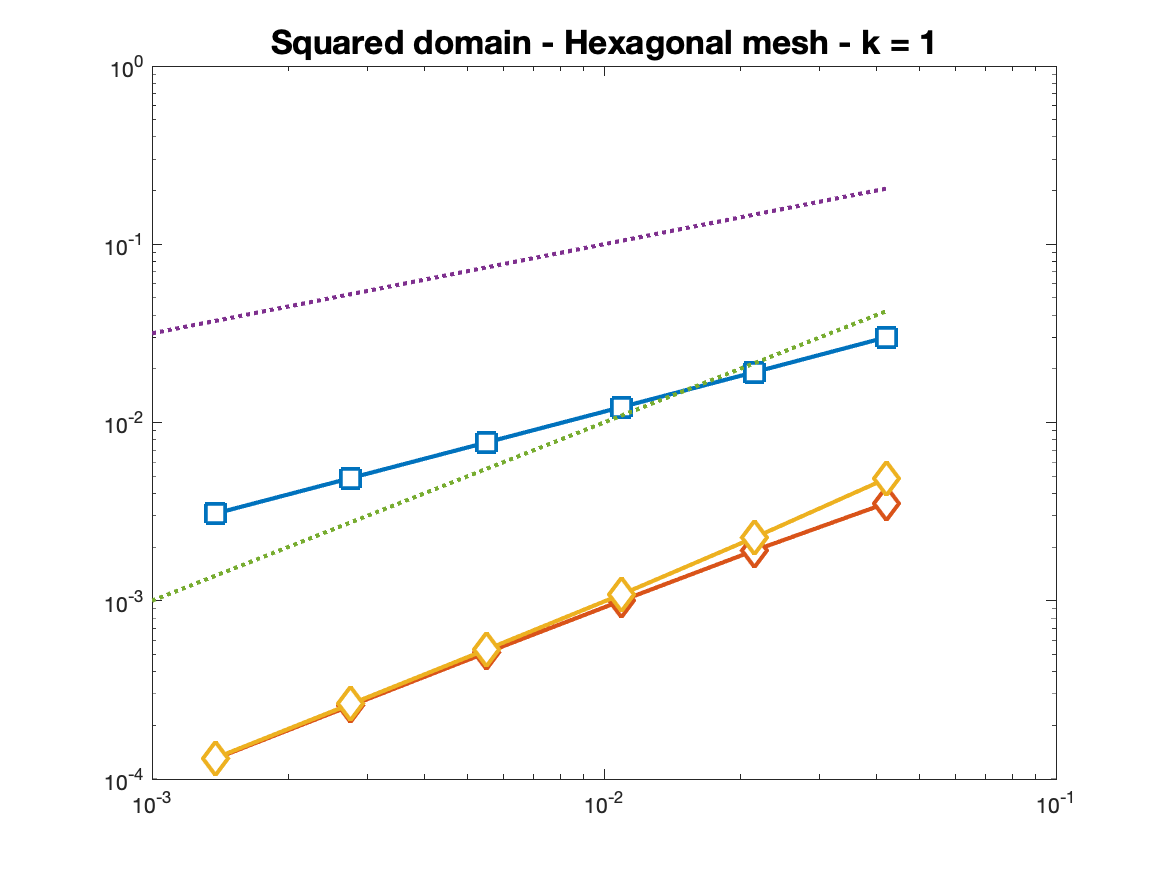}}\quad
	{\includegraphics[width=0.45\textwidth]{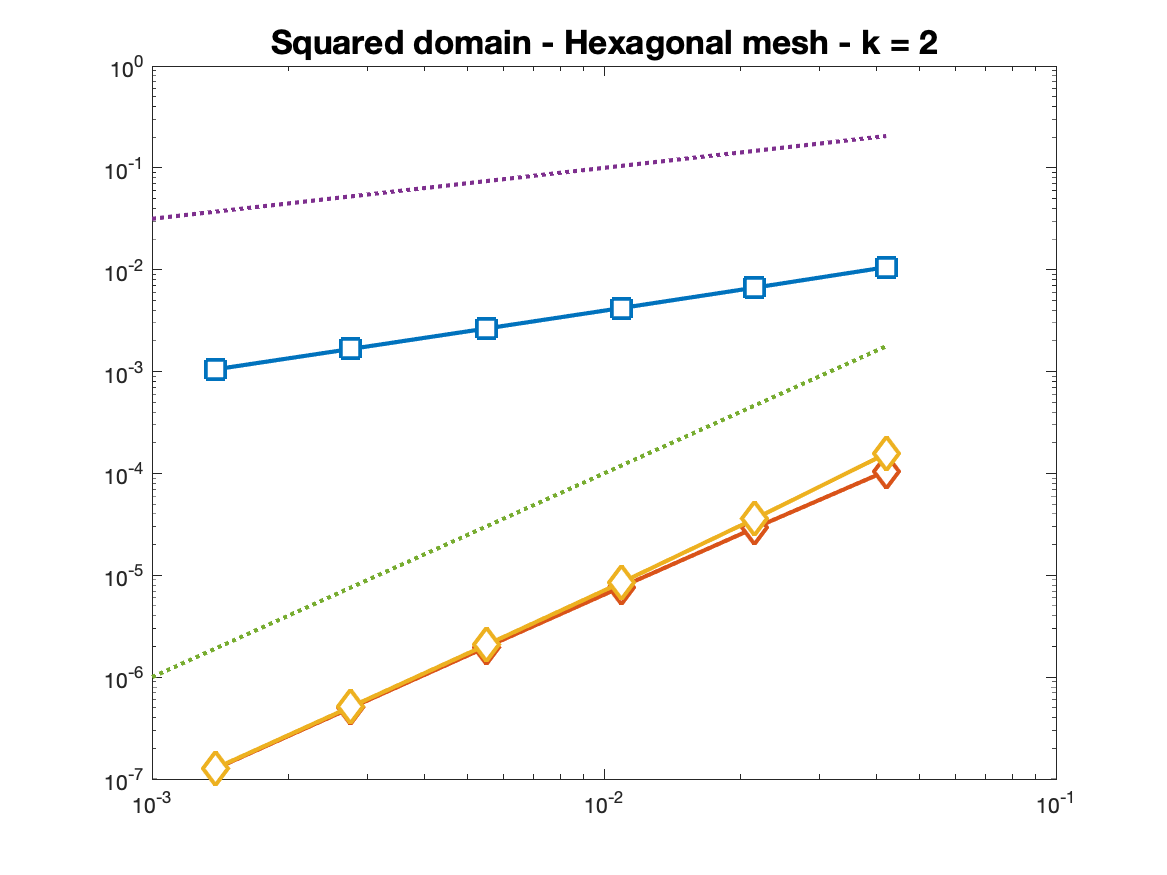}}\\
	{\includegraphics[width=0.45\textwidth]{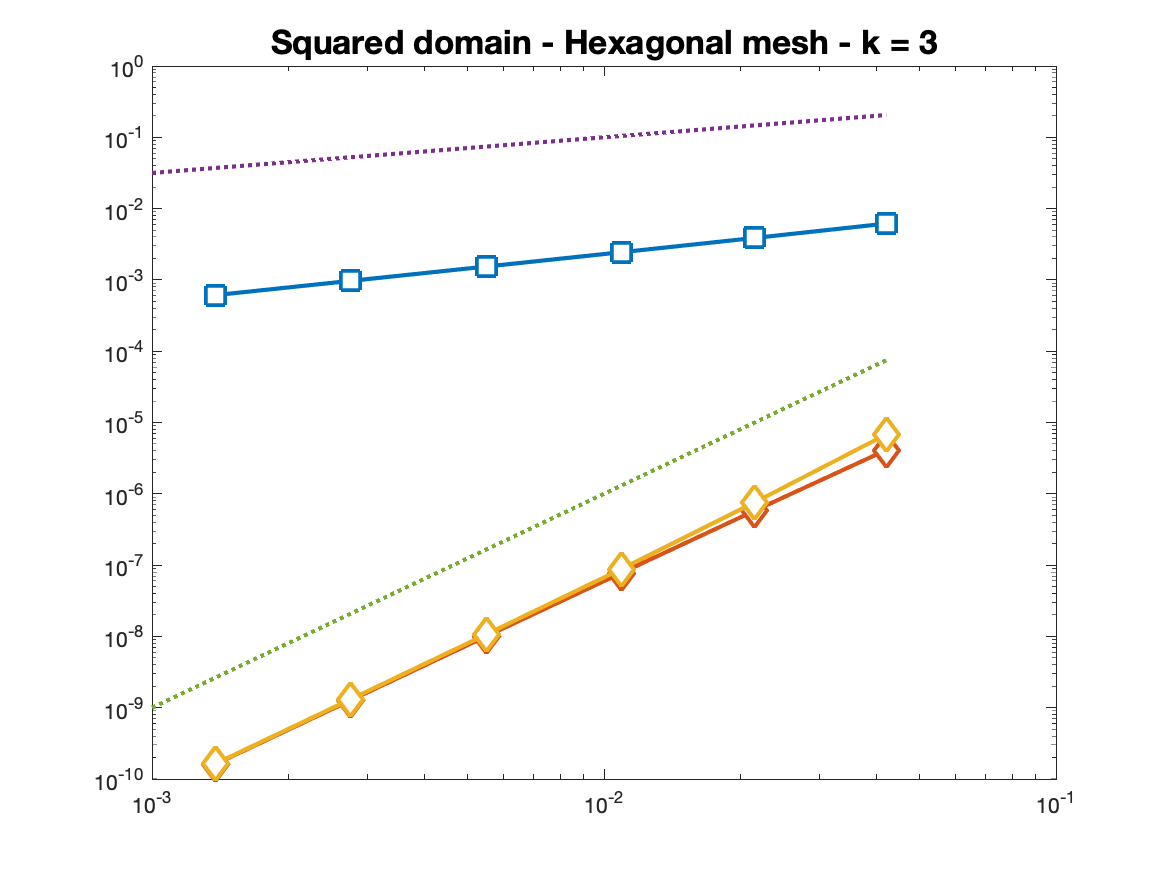}}\quad
	{\includegraphics[width=0.45\textwidth]{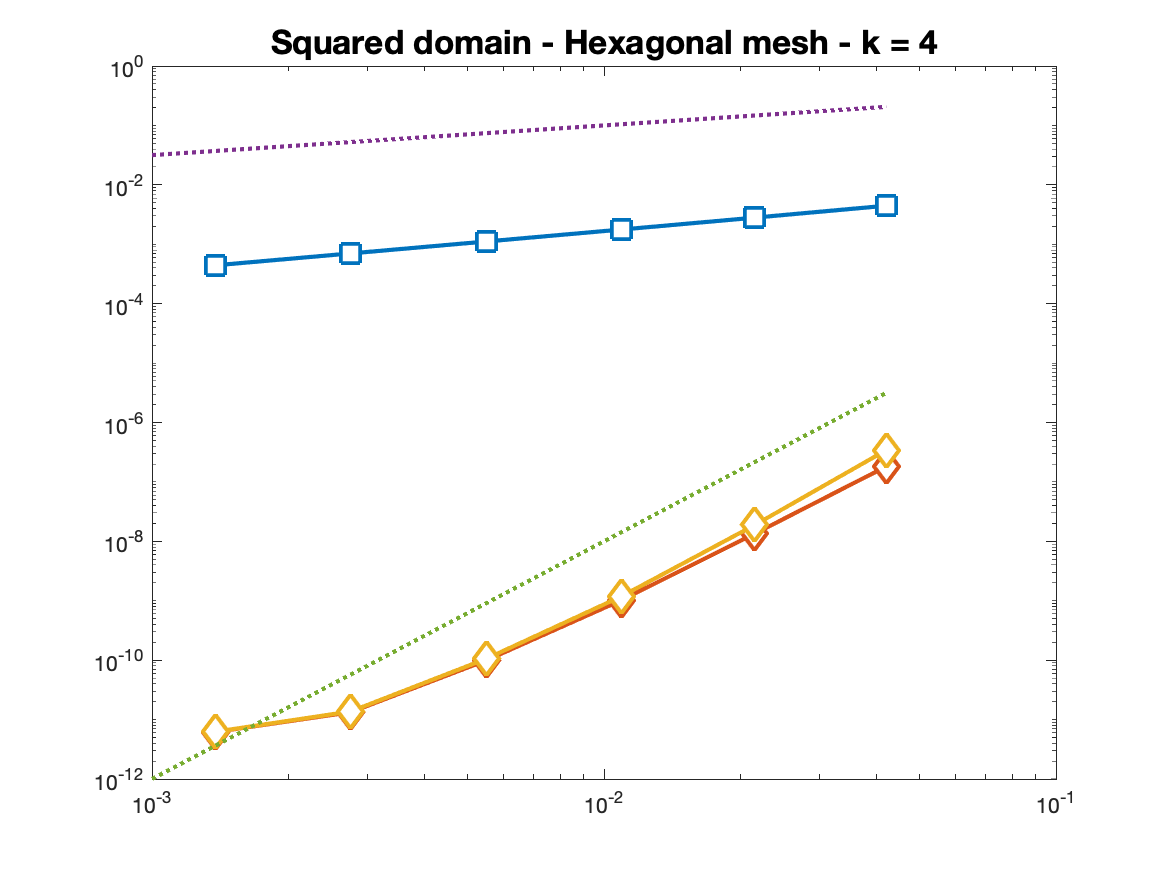}}
	\caption{
		Squared domain, hexagonal meshes. On the $x$ axis the meshsize $h$, on the $y$ axis the global (blue squares) and local (yellow/orange diamond) $H^1$ errors $e^1_\Omega$, 	$e^1_{\Omega_0^-}$ and 	$e^1_{\Omega_0^+}$. The slopes of the two reference lines are $1/2$ and $k$.
	}\label{fig:primotest}
	
\end{figure}

\begin{figure}
	\centering
	{\includegraphics[width=0.45\textwidth]{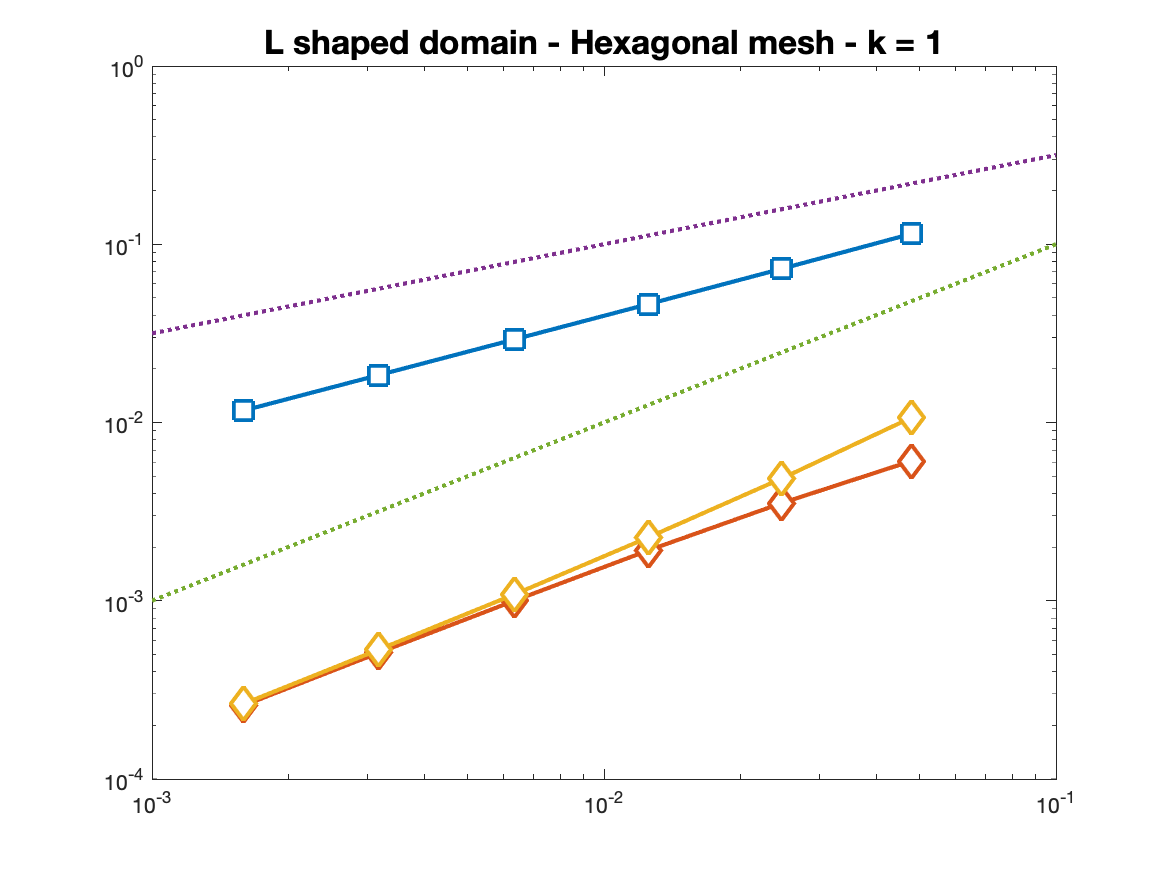}}\quad
	{\includegraphics[width=0.45\textwidth]{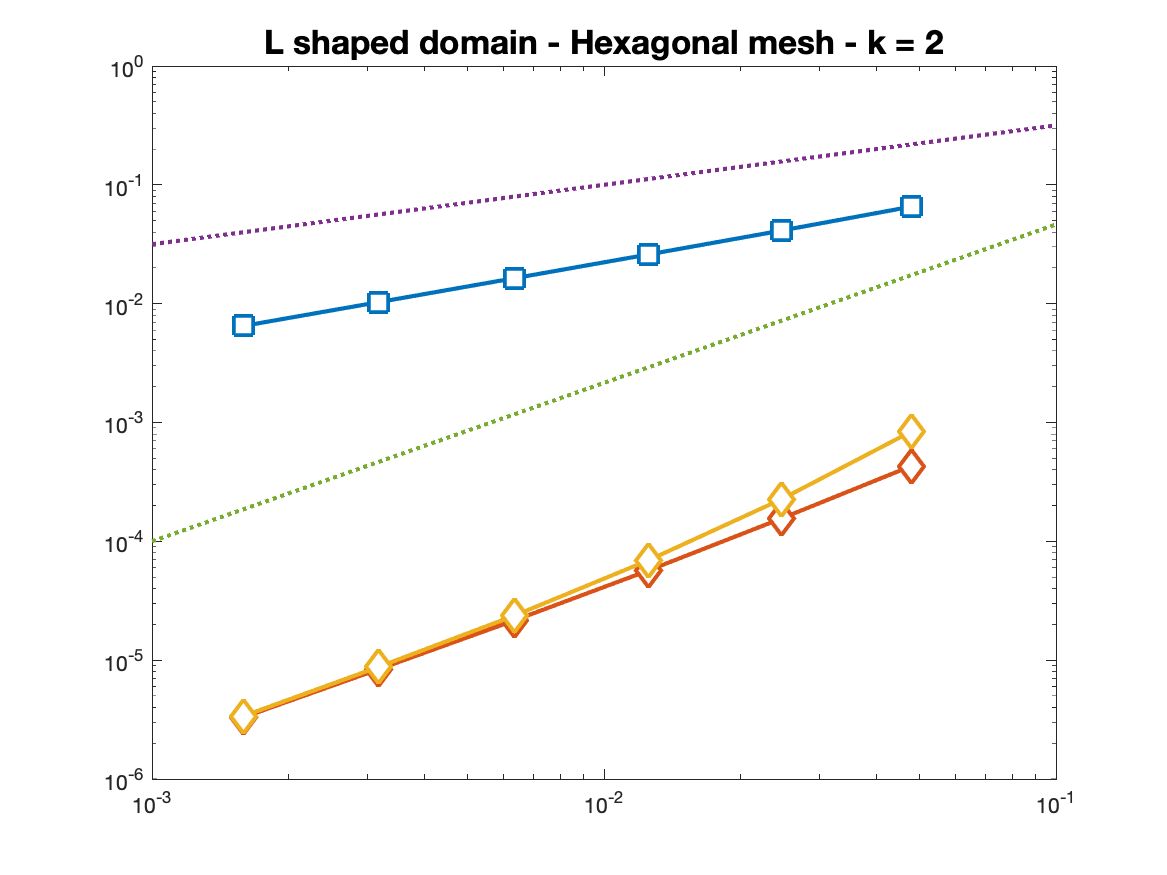}}\\
	{\includegraphics[width=0.45\textwidth]{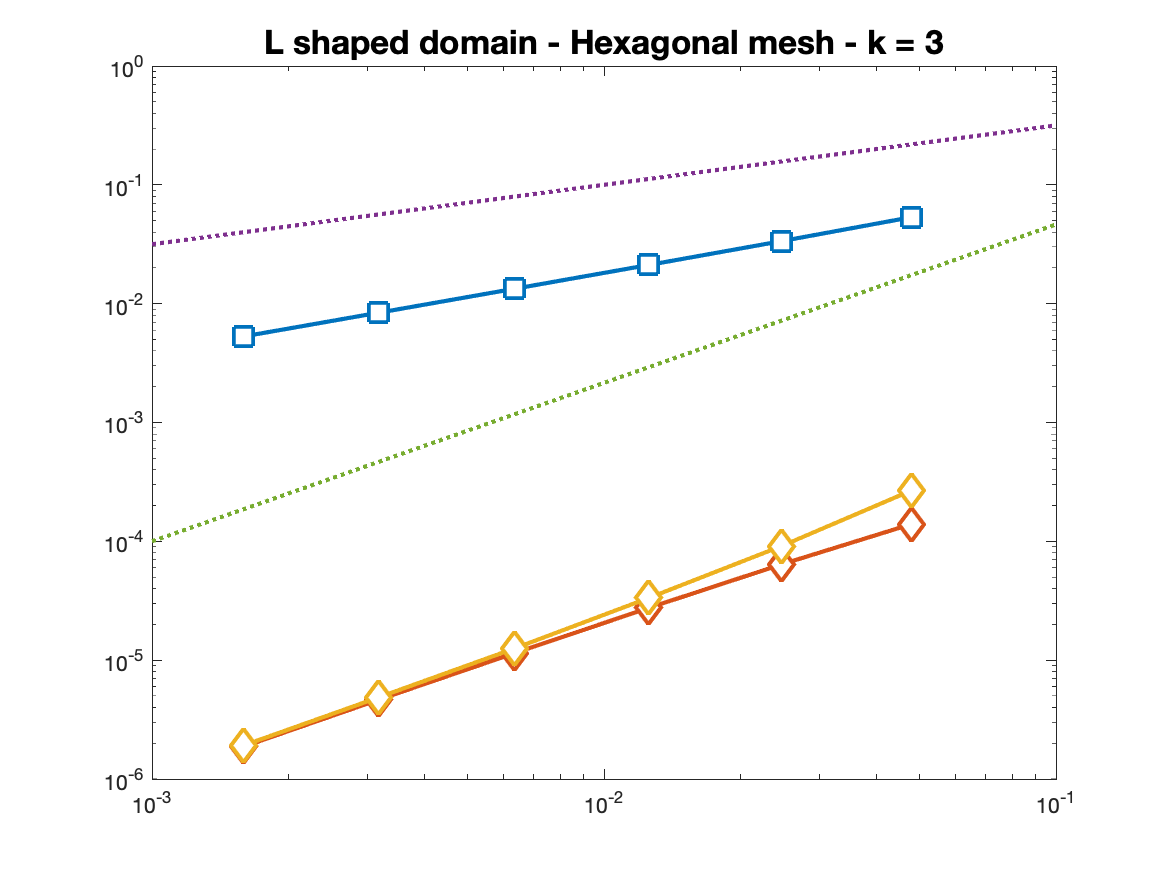}}\quad
	{\includegraphics[width=0.45\textwidth]{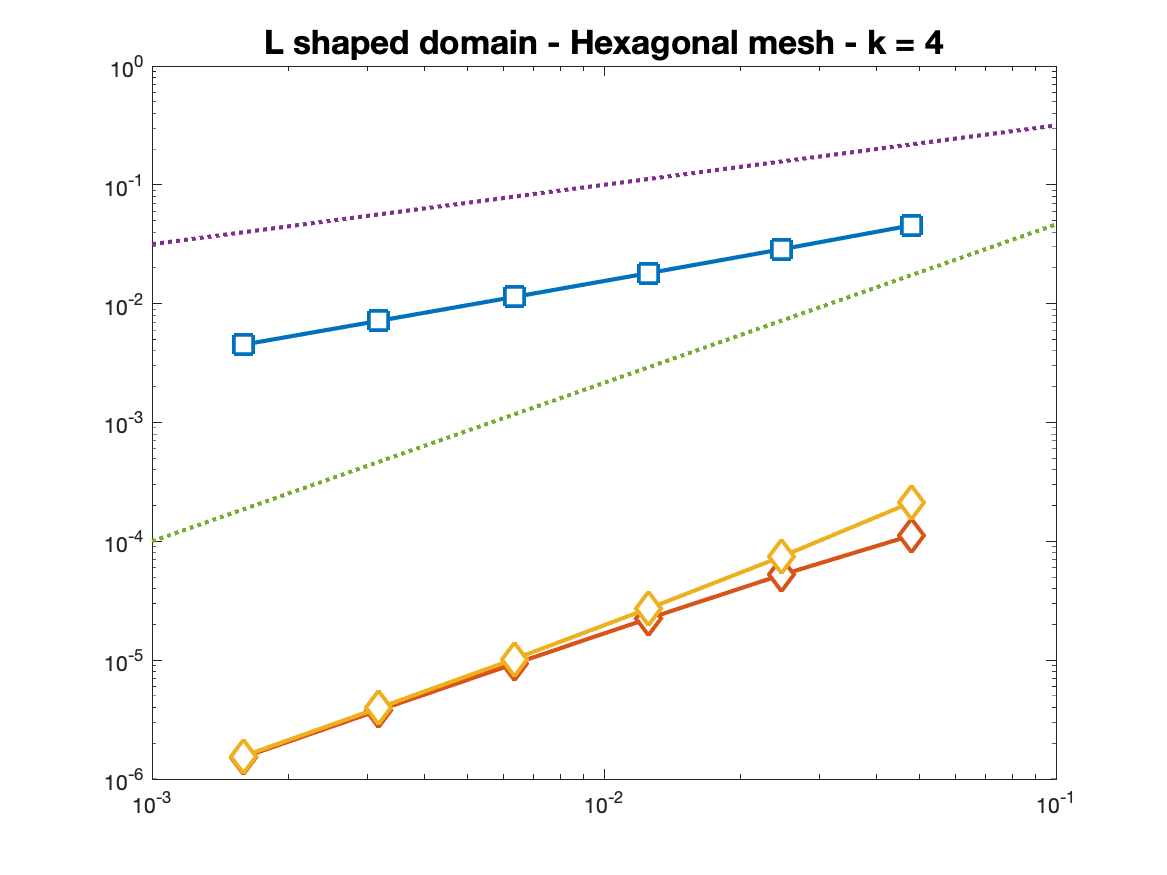}}
	\caption{
		L--shaped domain, hexagonal meshes. On the $x$ axis the meshsize $h$, on the $y$ axis the global (blue squares) and local (yellow/orange diamond) $H^1$ errors $e^1_\Omega$, 	$e^1_{\Omega_0^-}$ and 	$e^1_{\Omega_0^+}$. The slopes of the two reference lines are $1/2$ and $\min\{4/3,k\}$.
	}
\end{figure}

\begin{figure}
	\centering
	{\includegraphics[width=0.4\textwidth]{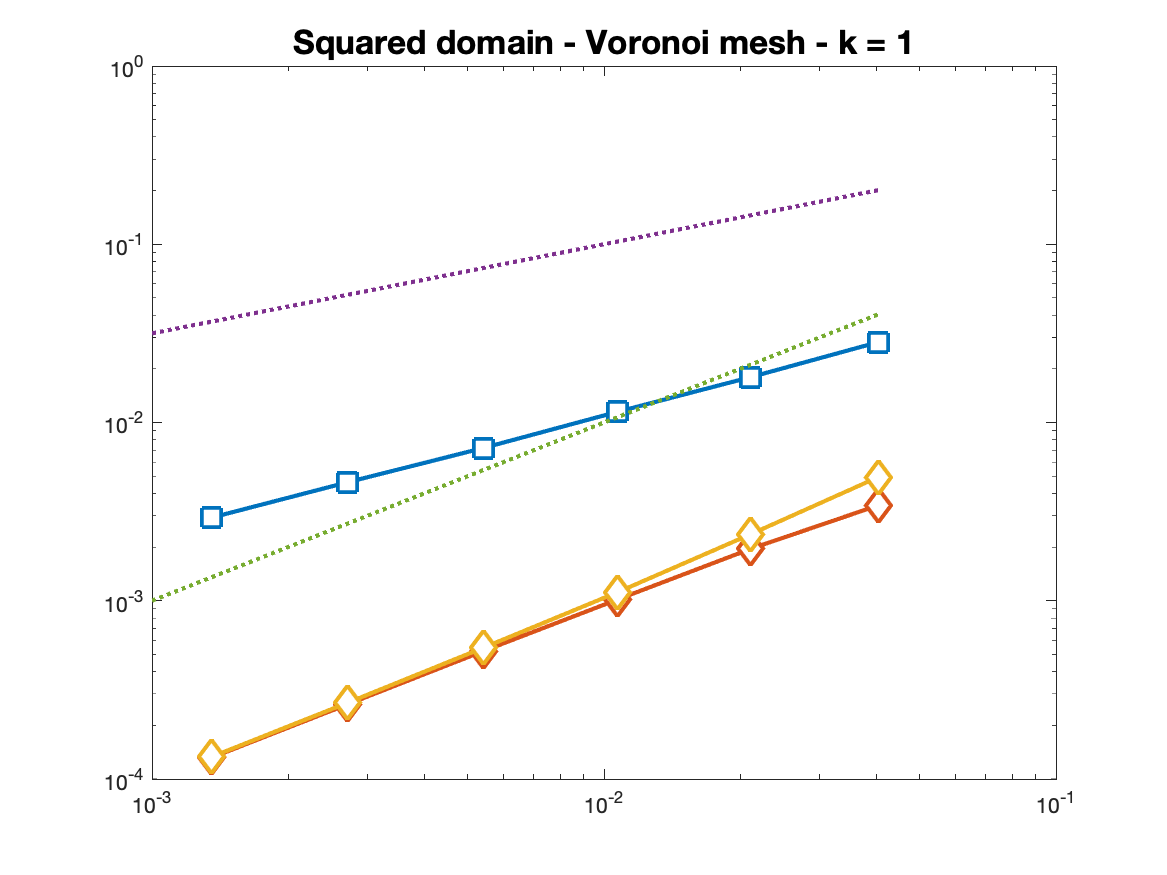}}\quad
	{\includegraphics[width=0.4\textwidth]{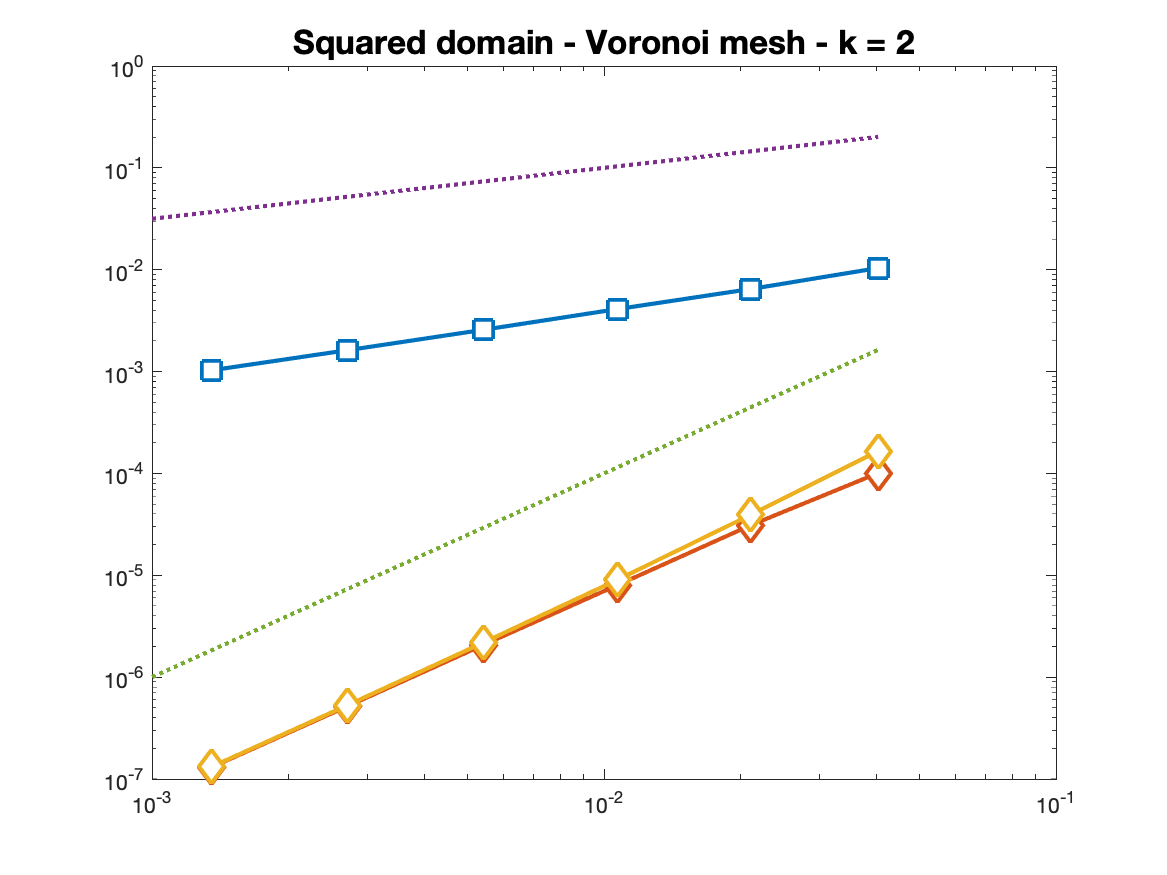}}\\
	{\includegraphics[width=0.4\textwidth]{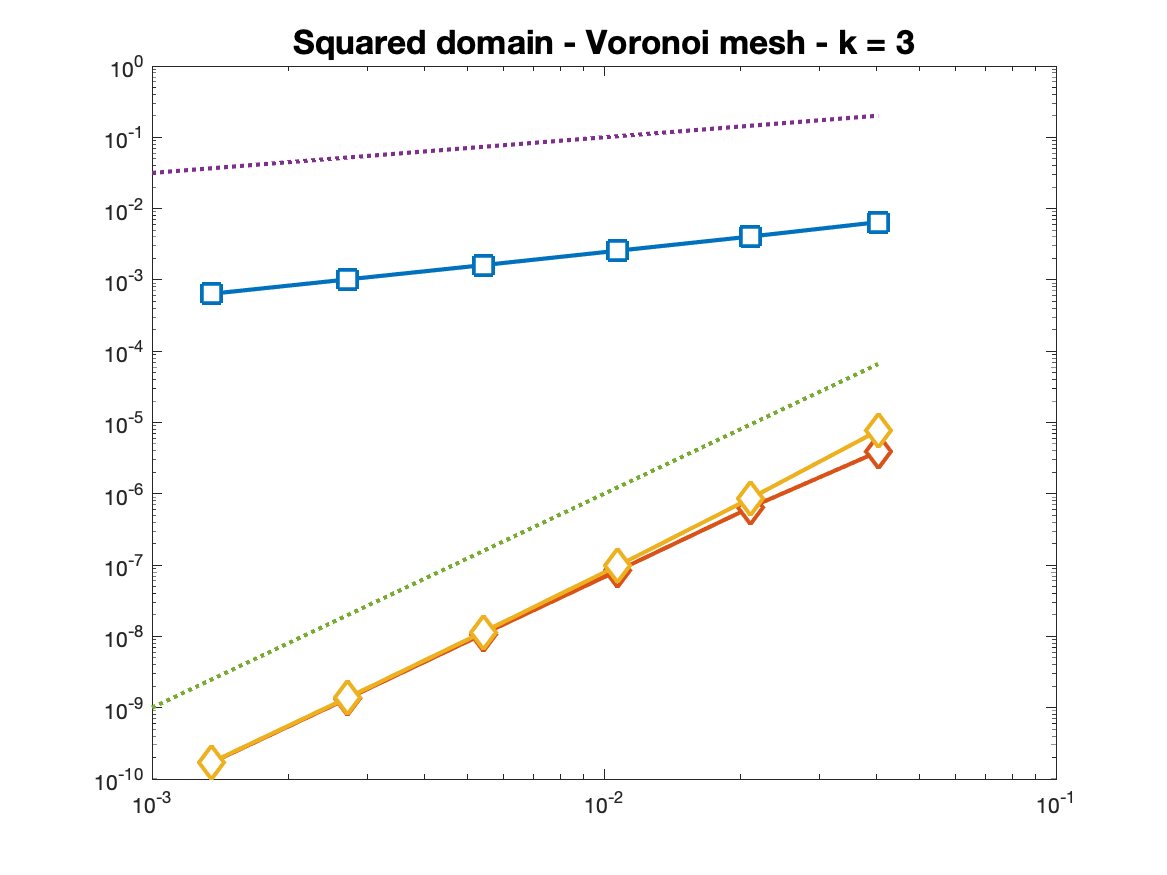}}\quad
	{\includegraphics[width=0.4\textwidth]{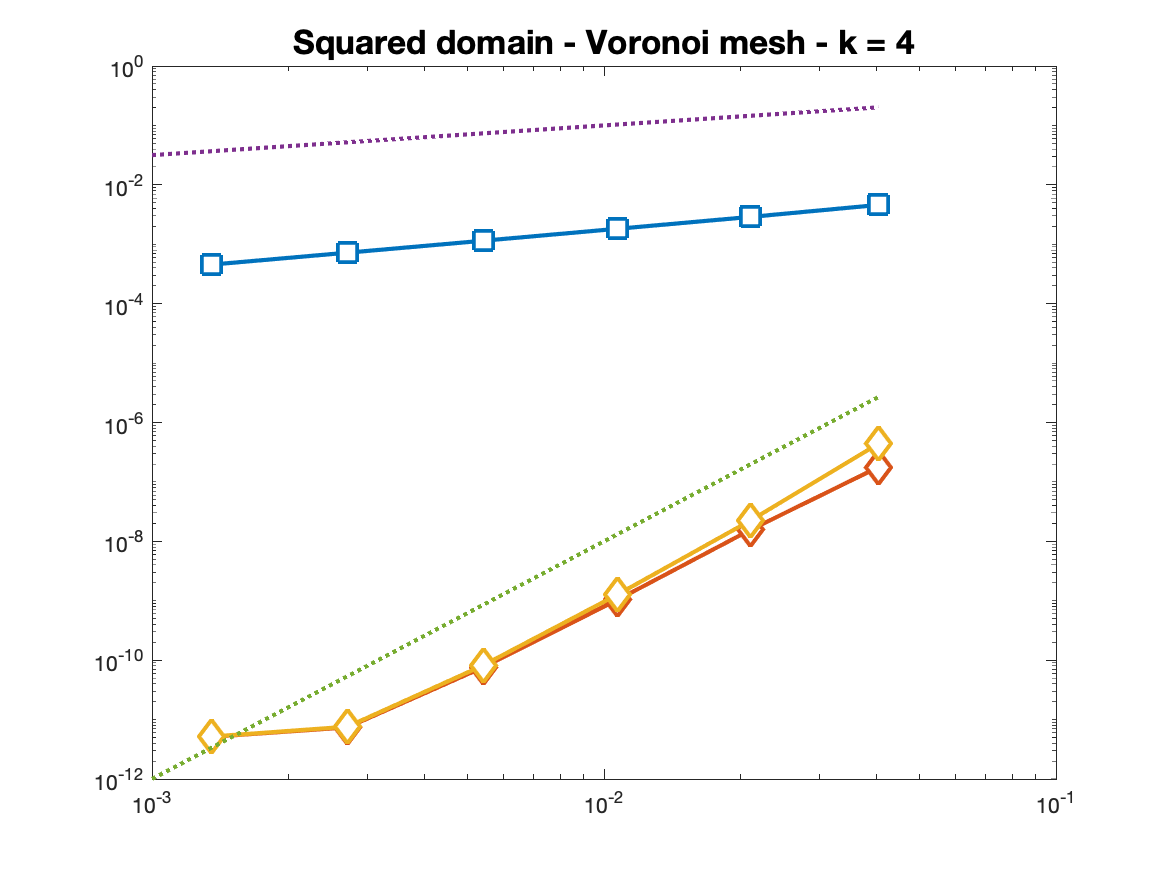}}	\caption{
		Squared domain, Voronoi meshes. On the $x$ axis the meshsize $h$, on the $y$ axis the global (blue squares) and local (yellow/orange diamond) $H^1$ errors $e^1_\Omega$, 	$e^1_{\Omega_0^-}$ and 	$e^1_{\Omega_0^+}$. The slopes of the two reference lines are $1/2$ and $k$.
	}
	
\end{figure}

\begin{figure}
	\centering
	{\includegraphics[width=0.4\textwidth]{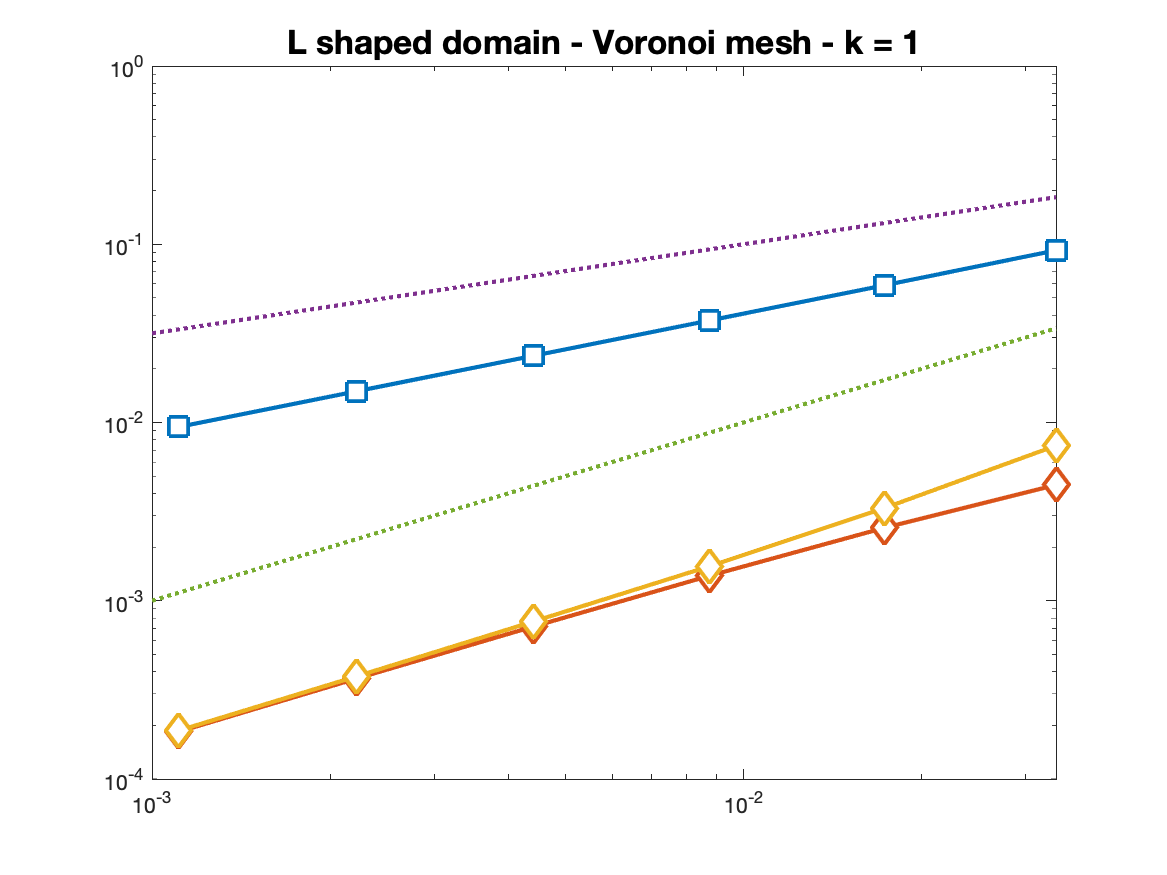}}\quad
	{\includegraphics[width=0.4\textwidth]{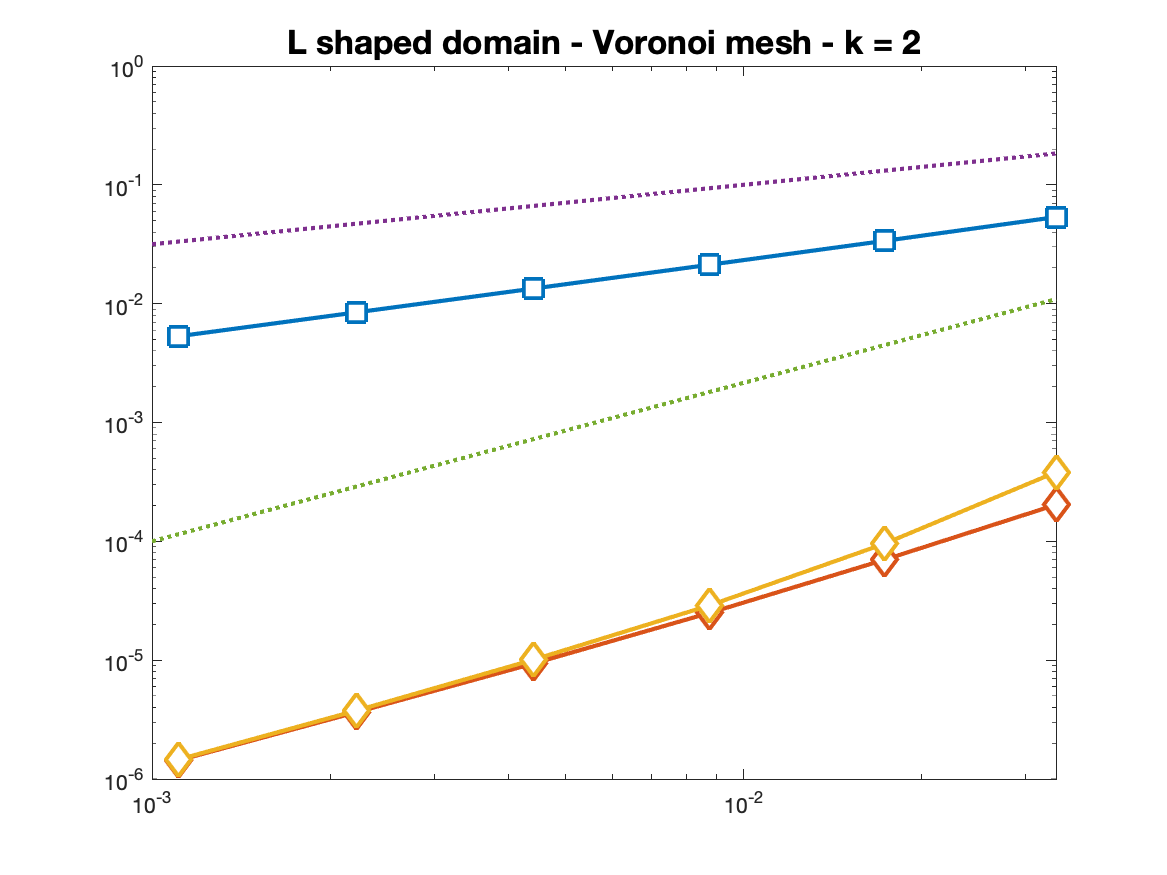}}\\
	{\includegraphics[width=0.4\textwidth]{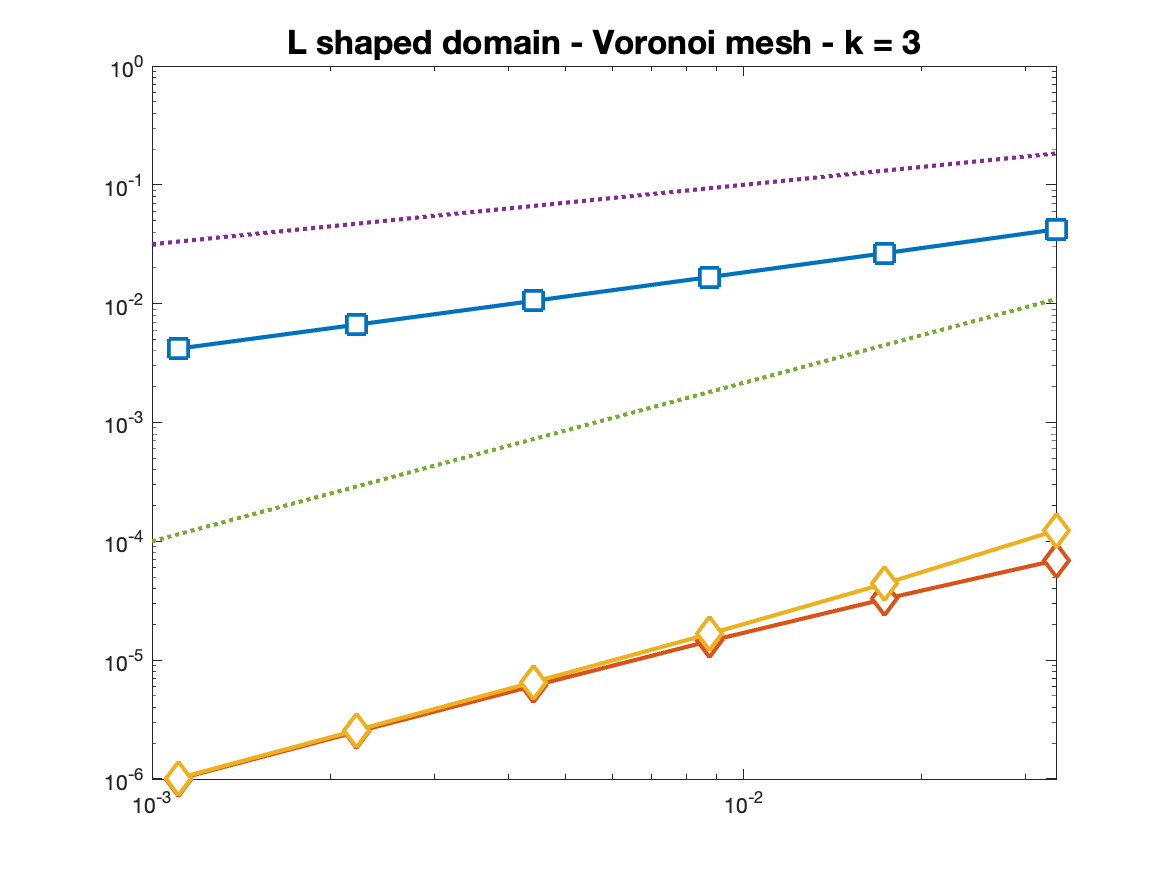}}\quad
	{\includegraphics[width=0.4\textwidth]{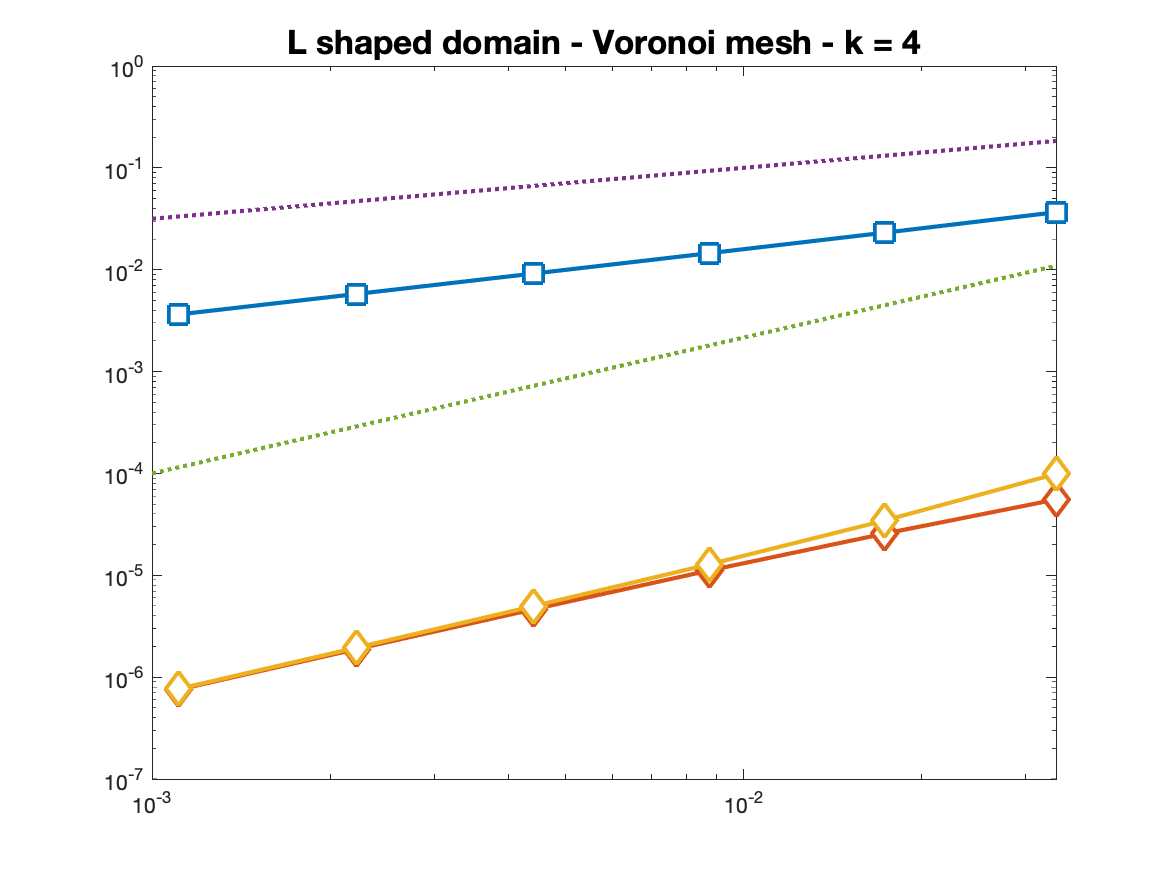}}
	\caption{
		L--shaped domain, Voronoi meshes.On the $x$ axis the meshsize $h$, on the $y$ axis the global (blue squares) and local (yellow/orange diamond) $H^1$ errors $e^1_\Omega$, 	$e^1_{\Omega_0^-}$ and 	$e^1_{\Omega_0^+}$. The slopes of the two reference lines are $1/2$ and $\min\{4/3,k\}$.
	}\label{fig:ultimotest}
\end{figure}

\newcommand{\Set}[1]{\{#1\}}
\section{Numerical tests}\label{sec:numerical}

In order to confirm the validity of the theoretical estimates we consider equation \eqref{Pb:strong} with the right hand side and Dirichlet data chosen in such a way that 
	\[
	u(x,y) = (x^2+y^2)^{1/3} \sin(2\theta/3),\qquad \theta = \tan^{-1}(y/x),
	% \cos(2\theta/3),\qquad \theta = \tan^{-1}(y/x),
	\]
	is the solution. 	We consider two different domains, namely
	\[
	\text{(Test 1)}\quad	\Omega = (0,1) ^2, \qquad \text{(Test 2)}\quad\Omega = (-1,1)^2 \setminus (0,1)^2
	\]
	The domain $\Omega_0 \subsubset \Omega$ on which we evaluate the error is chosen as
	\[\Omega_0 = \Set{(x,y) \in\mathbb R^2 : (x-0.5)^2+(y-0.5)^2 \leq 0.25^2} \quad \text{for Test 1}\]
	and 
	\[
	\Omega_0 = \Set{(x,y) \in\mathbb R^2 : (x+0.5)^2+(y+0.5)^2 \leq 0.25^2}\quad \text{for Test 2}.
	\] 
	The solution $u$  has a singularity in $(0,0)$, and,
	for both test cases, it verifies $u \in H^s(\Omega)$, for all $s < 3/2$, but $u \not \in H^{3/2}(\Omega)$.
	Consequently, we expect the global $H^1(\Omega)$ error not to converge faster than $h^{1/2}$. On the other hand the solution is smooth in a neighborhood of $\Omega_0$. According to Corollaries \ref{cor:finalsmooth} and \ref{cor:finalpoly} we can then expect, for Test 1 and Test 2 respectively, a convergence rate of order  $k$ and  $\min\{k,4/3-\varepsilon\}$ ($\varepsilon$ arbitrarily small).

% Basis for $\mathbb P_k(K)$: polynomials orthogonal with respect to the $L^2(K)$ scalar product.
\blu{We solve the problem by the enhanced virtual element method (see Section \ref{sec:enhanced}) with $k = 1, \cdots, 4$.} For both test cases we consider both a sequence of progressively fines structured hexagonal meshes and a sequence of progressively finer regular Voronoi meshes. Examples of the meshes used for the numerical tests are displayed in Figures \ref{fig:square-meshes} and \ref{fig:lshape-meshes}.  For all discretizations the stabilization is chosen to be the simple so called dofi--dofi stabilization, that, under our mesh regularity assumptions, is optimal.

In Figures \ref{fig:primotest} through \ref{fig:ultimotest} we plot, in logarithmic scale, the convergence history for the two test cases and the two sequences of meshes. For $k = 1$ through $4$ we plot the global error (square markers) as well as the local error (diamond markers). In order to avoid the need of evaluating  integrals over a curved domain, rather than displaying the actual value of the local error, we display upper and lower approximations obtained by evaluating the errors in subdomains $\Omega_0^- \subset \Omega_0 \subset \Omega_0^+$ defined as 
\[
 \Omega_0^- = \cup_{K \in \Tess^-} \bar K, \qquad \Omega_0^+ = \cup_{K \in \Tess^+} \bar K
\]
with
\[
\Tess^- = \{ K \in \Tess: K\subseteq \Omega_0\} \quad \text{ and }\quad  \Tess^+ = \{ K \in \Tess: K\cap \Omega_0\not = \emptyset\}. 
\]

We then set 
	\begin{gather}
		e^1_{\Omega} = (||u - {\Pi^0 _k} u_h||^2_{0,\Omega} + ||\nabla u - {\Pi^0 _{k-1}}\nabla u_h||^2_{\Omega})^{1/2},\\
	e^1_{\Omega_0^-} = (||u - {\Pi^0 _k} u_h||^2_{0,\Omega_0^-} + ||\nabla u - {\Pi^0 _{k-1}}\nabla u_h||^2_{\Omega_0^-})^{1/2},\\
		e^1_{\Omega_0^+} = (||u - {\Pi^0 _k} u_h||^2_{0,\Omega_0^+} + ||\nabla u - {\Pi^0 _{k-1}}\nabla u_h||^2_{\Omega_0^+})^{1/2}.
		\end{gather}
 In all figures, we display, for reference purpose, dotted straight lines with a slope corresponding to the expected convergence rate for global and local error, namely $1/2$ for the global error and, respectively, $k$ and $\min\{4/3,k\}$ for the local error in the squared and in the L--shaped domain.

Figures \ref{fig:primotest} through \ref{fig:ultimotest} clearly confirm the validity of the a priori estimate. In particular the local error behaves as expected both in the squared domain case and in the L--shaped domain case.

\bibliographystyle{amsplain}

\bibliography{localEstimates}

\end{document}

%% file: localEstimates_def.tex
\newcommand{\polygon}{K}
\renewcommand{\K}{\polygon} %elemento poligonale
\renewcommand{\epsilon}{\varepsilon}

\newcommand{\G}{\Omega}
\newcommand{\Guno}{{\G_1}}

\newcommand{\tu}{\widetilde u}
\newcommand{\cu}{\widehat u_h}

\newcommand{\mesh}{\mathcal{T}_h}

\newcommand{\xK}{{\mathbf x}_\K}
\newcommand{\rK}{\rho_\K}
\newcommand{\hK}{h_\K}

\newcommand{\EdgesP}{\mathcal{E}_\K}

\newcommand{\gstar}{\gamma_0}
\newcommand{\guno}{\gamma_1}
\renewcommand{\phi}{\varphi}

\newcommand{\gbP}{\gamma_2}

\newcommand{\Nstar}{N^\star}
\newcommand{\mapK}{F}
\newcommand{\unitcircle}{\mathcal{B}_1}
\newcommand{\family}{\mathcal{F}}

\newcommand{\bK}{\partial\K}
\newcommand{\BoK}{\mathbb{B}_k(\bK)}
\newcommand{\Tess}{\mathcal{T}_h}
\newcommand{\PinablaK}{\Pi^{\nabla}_{\K}}
\newcommand{\Pinabla}{\Pi^{\nabla}}
\newcommand{\aK}{a_{\K}}

\newcommand{\VEMK}{V^k_m(\K)}
\newcommand{\Vh}{V_h}
\newcommand{\stabK}{s^{\K}}
\newcommand{\dofi}[1]{\mathrm{dof}^K_i(#1)}
\newcommand{\uh}{u_h}
\newcommand{\vh}{v_h}

\newcommand{\NdofK}{N_K}
\newcommand{\sK}{s_K}

\newcommand{\tPinablaK}{\widetilde \Pi^\nabla_K}
\newcommand{\QnablaK}{Q^\nabla_K}
\newcommand{\tPinabla}{\widetilde \Pi^\nabla}
\newcommand{\Qnabla}{Q^\nabla}
\newcommand{\Vhstar}{{\widetilde V_h}}

\newcommand{\IK}{\mathfrak{I}_K}

\newcommand{\weight}{\omega}
\newcommand{\resto}{\rho^K}

\newcommand{\tf}{\widetilde f}
\newcommand{\Dh}{\Delta_h}
\newcommand{\DK}{\Delta^K}
\newcommand{\df}{\delta_f}

\newcommand{\subsubset}{\subset\subset}
\newcommand{\vphi}{v^{\phi}}
\renewcommand{\O}{{\Omega}}

\newcommand{\wh}{w_h}
\newcommand{\Komega}{\mathcal{K}_{\omega}}

\newcommand{\te}{\widetilde e}
\newcommand{\ce}{\widehat e_h}
\newcommand{\Vho}{\mathring V_h}

\newcommand{\Ih}{I_h}
\newcommand{\fIh}{\mathfrak{I}_h}

\newcommand{\Pstar}[1]{\widetilde{\mathbb{P}}_{#1}(\Tess)}
\newcommand{\Pizero}[1]{\Pi^0_{#1}}
\newcommand{\fh}{f_h}

\newcommand{\RR}{\mathbb{R}}
\newcommand{\Poly}[1]{\mathbb{P}_{#1}}

\newcommand{\pim}{\Pi^0_m}

\renewcommand{\g}{\gamma}
\renewcommand{\s}{\tau}

\newcommand{\hOmega}{\widehat\Omega}

\newcommand{\sk}{\kappa}
\newcommand{\sm}{\mu}

\newcommand{\Id}{{\rm I}}
\newcommand{\uno}{\Id}

%% file: localnegative.tex
\subsection{Local  negative norm error estimates} \label{sec:negativelocal}

For $\Omega$ being either a polygonal domain, or a smooth domain discretized by elements where, exclusively for elements adjacent to the boundary, curved edges are allowed, we now prove some bounds on the local error, measured in negative norms. We start by proving the following lemma.

\begin{lemma}\label{lem:6b} Let $\G_0 \subsubset \Guno \subsubset \Omega$ be  fixed smooth interior subdomains of  $\Omega$, and let the solution $u$ to Problem \eqref{Pb:strong} satisfy $ u \in H^{t+1}(\Guno)$, with $1 \leq t \leq k$. Let $u_h \in \Vh$ denote the solution to Problem \eqref{pb:discrete}.  Then, there exists an $h_0$ such that, if $h < h_0$ we have, for $p \geq 0$ integer, 
	\[
	\| u - u_h \|_{-p,\Omega_0} \lesssim h^\g | u - u_h |_{1,\Guno} + \| u - u_h \|_{-p-1,\Guno} +  h^{\s + t} |  u |_{t+1,\Guno},	\] 
with $\g = \min\{ p+1,k\}$, and $\s = \min\{p+1,m\}$.
\end{lemma}

\begin{proof}
	Let $\G'$ with $\G_0 \subsubset \G' \subsubset \Guno$, be a fixed intermediate subdomain between $\G_0$ and $\Guno$,  and let $\weight \in C^{\infty}_0(\G')$ with $\weight = 1$ in $\G_0$.
		We let $h_0$ be such that for all $h < h_0$, all elements $K \in \Tess$ with $K \cap \G' \not= \emptyset$ satisfy $K \subset \overline{\Omega}_1$, and we let $h < h_0$.		
	 Letting $e = u-u_h$, we have
	\[
	\| e \|_{-p,\G_0} \leq \| \weight e \|_{-p,\Guno} = \sup_{\phi\in H^p_0(\Guno)} \frac {\int_\Guno \weight e \phi} {\| \phi \|_{p,\Guno} }	= \sup_{\phi\in H^p_0(\Guno)} \frac {\int_\Guno \nabla \weight e \cdot \nabla \vphi} {\| \phi \|_{p,\Guno} },
	\]
	where $\vphi$ is the solution to 
	\[
	-\Delta \vphi = \phi, \quad \text{ in }\Guno, \qquad \vphi = 0, \quad \text{ on } \partial \Guno.
	\]
	Observe that, as we assumed that $\Guno$ is smooth, $\phi \in H^p(\Guno)$  implies $\vphi \in H^{p+2}(\Guno)$, with
	\begin{equation}\label{eq:vphismooth}
	\| \vphi \|_{p+2,\Guno} \lesssim \| \phi \|_{p,\Guno}.
	\end{equation}
Using the error equation  \eqref{othererrorequation} for $\vh \in \Vh$ with $\supp \vh \subseteq \overline{\G}_1$ arbitrary, we can write
	\begin{multline}\label{eq:int1}
	\int_\Guno \nabla(\weight e) \cdot \nabla \vphi = \int_\Guno \nabla e \cdot \nabla(\weight \vphi) + \int_\Guno e [
	\nabla \weight \cdot \nabla \vphi + \nabla \cdot (\vphi\nabla \weight)
	] \\= \int_\Guno \nabla e\cdot \nabla (\weight \vphi - \vh) + \int_\Guno \nabla e \cdot \nabla \vh +  \int_\Guno e [
	\nabla \weight \cdot \nabla \vphi + \nabla \cdot (\vphi\nabla \weight)]\\
	 = \int_\Guno \nabla e\cdot \nabla (\weight \vphi - \vh) + \Dh(e,\vh) + \int_\Guno \df \vh - \Dh(u,\vh) +   \int_\Guno e [
	 \nabla \weight \cdot \nabla \vphi + \nabla \cdot (\vphi\nabla \weight)]\\
	 = I + II + III + IV + V.
	\end{multline}
	We now let $\vh = \fIh (\weight \vphi)$ and we remark that, as $\supp(\weight \vphi) \subset \G' \subsubset \Guno$, if $h < h_0$, then $\supp(\vh) \subset \overline{\G}_1$. We bound the five terms on the right hand side of \eqref{eq:int1} separately. Using \eqref{interperror} we can write, with $\g = \min\{p+1,k \}$
	\begin{multline}
	I = \int_\Guno \nabla e\cdot \nabla (\weight \vphi - \fIh(\weight \vphi )) \leq | e |_{1,\Guno} | \weight \vphi - \fIh(\weight \vphi ) |_{1,\Guno}
\\	 \lesssim | e |_{1,\Guno}\, h^\g | \weight \vphi |_{p+2,\Guno} \lesssim |  e |_{1,\Guno}\, h^\g \|  \vphi \|_{p+2,\Guno}\lesssim h^\g | e |_{1,\Guno} \| \phi \|_{p,\Guno}.
	\end{multline}
	Adding and subtracting $\weight \vphi$ and using \eqref{interperror} and \eqref{errorDeltaK} we have
	\begin{multline*}
	II =   \Dh(e,\fIh (\weight \vphi)) =  \Dh ( e , \fIh (\weight \vphi) - \weight \vphi) + \Dh(e,\weight\vphi)\\[2mm]
\lesssim 	| e |_{1,\Guno} | \fIh (\weight \vphi) - \weight \vphi |_{1,\Guno} + h^{\g}_K | e |_{1,\Guno} | \weight \vphi |_{p+2,\Guno} \\[1.8mm]
	 \lesssim h^\g| e |_{1,\Guno} | \weight \vphi |_{p+2,\Guno}  \lesssim h^\g | e |_{1,\Guno} \|  \vphi \|_{p+2,\Guno} \lesssim  h^\g | e |_{1,\Guno} \| \phi \|_{p,\Guno}.	
	\end{multline*}
Moreover, using the fact that $f - \Pi^0_m f$ is orthogonal to $\Pstar{m} \ni \Pi^0_m(\fIh (\weight \vphi))$, we can write, with $\G_h \subseteq \Guno$ denoting the union of elements $K \in \Tess$ such with $\fIh(\weight \vphi) \not = 0$ in $K$,
		\begin{multline}
III = \int_\Guno \df \fIh(\weight \vphi) =
\int_{\G_h} (f - \Pi^0_m f) (\fIh (\weight \vphi) - \Pi_m^0 \fIh( \weight \vphi))
 \\ \lesssim  
\|  f - \Pi^0_m f \|_{0,\G_h} \left(\| (\uno-\Pi_m^0)(\fIh (\weight \vphi) - \weight \vphi)  \|_{0,\Guno} + \|  \weight \vphi - \Pi_m^0  (\weight \vphi)\|_{0,\Guno} \right)\\[1.8mm]
\lesssim \|  f - \Pi^0_m f \|_{0,\G_h} \left(\| \fIh (\weight \vphi) - \weight \vphi  \|_{0,\Guno} + \|  \weight \vphi - \Pi_m^0  (\weight \vphi)\|_{0,\Guno} \right)
\end{multline}
where we added and subtracted $\weight\vphi - \Pi^0_m(\weight \vphi)$.
Using \eqref{eq:polyapproxK} and \eqref{interperror} gives 
 \begin{multline}\label{eq:5.13}
III		\lesssim h^{t-1} | f |_{t-1,\G_h} \left(
		h \| \fIh (\weight \vphi)  - \weight \vphi \|_{1,\Guno} + h^{\s +1} \| \weight \vphi \|_{p+2,\Guno}
		\right)\\[1.8mm]
		\lesssim h^{t-1} | f |_{t-1,\Guno} \left(
		h^{\g+1} \| \weight \vphi \|_{p+2,\Guno} + h^{\s+1} \| \weight \vphi \|_{p+2,\Guno}
		\right)	 \lesssim h^{t + \s} | f |_{t-1,\Guno} \| \phi \|_{p,\Guno}
		\end{multline}	
(we recall that $f = -\Delta u$ so, under our assumptions, we have that $f \in H^{t-1}(\Guno)$). By once again adding and subtracting $\weight \vphi$ and using \eqref{eq:comm1} and \eqref{errorDeltaK} we have
	\begin{equation*}
		IV = - \Dh(u,\vh) \lesssim h^{t+\g} | u |_{t+1,\Guno} \| \weight \vphi \|_{p+2,\Guno} \lesssim h^{t+\g} | u |_{t+1,\Guno} \| \phi \|_{p,\Guno}.
	\end{equation*}
Finally, we bound $V$ as in \cite{NS} as
\begin{multline*}
	V = \int_\Guno e [
	\nabla \weight \cdot \nabla \vphi + \nabla \cdot (\vphi\nabla \weight)]  \lesssim \| e \|_{-p-1,\Guno} \|  \nabla \weight \cdot \nabla \vphi + \nabla \cdot (\vphi\nabla \weight) \|_{p+1,\Guno} \\
	\lesssim \| e \|_{-p-1,\Guno} \| \vphi \|_{p+2,\Guno} \lesssim \| e \|_{-p-1,\Guno} \| \phi \|_{p,\Guno}.
\end{multline*}
The thesis follows from the observation that $\s \leq \g$ and $\| f \|_{t-1,\Guno} \lesssim \| u \|_{t+1,\Guno}$.
	\end{proof}
	
	\begin{remark} We observe that, as $f \in L^2(\Omega)$, we have that $u|_{\G_1} \in H^2(\G_1)$ for all $\G_1 \subsubset \Omega$. Then the assumptions of Lemma \ref{lem:6b} are always satisfied for some $t \geq 1$.
	\end{remark}
	
	\begin{remark}
	In the special case in which the right hand side of \eqref{Pb:weak} is computable for all $\vh \in V_h$ with $\supp v \subseteq \bar\Omega_1$ (this happens if $f|_{\Guno}$ is a polynomial of degree at most $m$) then the term $III$ in the sum at the right hand side of \eqref{eq:int1} vanishes. In such a case, we have a better estimate, namely
		\[
		\| u - u_h \|_{-p,\Omega_0} \lesssim h^\g | u - u_h |_{1,\Guno} + \| u - u_h \|_{-p-1,\Guno} +  h^{\g  + t} | u |_{t+1,\Guno}.	\] 
	\end{remark}
	
%	\begin{remark}
%		We observe that if $d\in H^1(\Omega)$ satisfies \[
%		a_h (d,\vh) =0, \quad \forall \vh \in \Vh
%		\]
%then, by a simila approach we can prove that
%\[
%\| d \|_{-p,G_0} \lesssim h | d |_{1,\Guno} + \| d \|_{-p-1,\Guno}.
%\]	
%	\end{remark}
	
	\

A recursive application of Lemma \ref{lem:6b} yields the following lemma.	
	\begin{lemma}\label{lem:7b}
Under the assumption of Lemma \ref{lem:6b}, for $p > 0$ arbitrary integer, there exists $h_0 > 0$ such that, provided $h<h_0$ we have
		\[
		\| e \|_{0,\G_0} \lesssim h | e |_{1,\Guno} + h^{t} | u |_{t+1,\Guno} + \| e \|_{-p,\Guno}.
		\]
	\end{lemma}
\begin{proof} Let $\hOmega_\ell$, $\ell = 0,\cdots,p$ be an increasing sequence of intermediate subdomains with
	\(	\G_0 = \hOmega_0 \subsubset \hOmega_1 \subsubset \cdots \subsubset \hOmega_p = \Guno
	\). By Lemma \ref{lem:6b}, for $\ell  = 0,\cdots,p-1$, there exists $h_{0,\ell}$ such that, provided $h < h_{0,\ell}$, 
 it holds that
 \[
 \| u - u_h \|_{-\ell,\hOmega_\ell} \lesssim h^{\g_\ell} \| u - u_h \|_{1,\hOmega_{\ell+1}} + \| u - \uh \|_{-\ell-1,\hOmega_{\ell+1}} + h^{\s_\ell+t} | u |_{t+1,\hOmega_{\ell+1}}
 \]	
 where $\g_\ell = \min\{\ell+1,k\}$, $\s_\ell = \min\{\ell+1,m\}$. Then, if $h < h_0 = \min_{\ell} h_{0,\ell}$ we can write
 \begin{multline}
 \| u - u_h \|_{0,\hOmega_0} \lesssim \| u - \uh \|_{-1,\hOmega_{1}} + h^{\g_0} \| u - u_h \|_{1,\hOmega_{1}} + h^{\s_0+t} | u |_{t+1,\hOmega_{1}}\\[1.8mm]
\lesssim \| u - \uh \|_{-2,\hOmega_{2}} + (h^{\g_1} + h^{\g_0})\| u - u_h \|_{1,\hOmega_{2}} + h^t (h^{\s_1} + h^{\s_0}) | u |_{t+1,\hOmega_{2}}\\
\lesssim \cdots 
\lesssim \| u - \uh \|_{-p,\hOmega_{p}} + (\sum_{\ell=0}^{p-1} h^{\g_\ell})\| u - u_h \|_{1,\hOmega_{p}} + h^t (\sum_{\ell=0}^{p-1} h^{\s_\ell}) | u |_{t+1,\hOmega_{p}}.
 \end{multline}
 	We conclude by remarking that, since $\g_\ell = \min \{\ell+1,k \} \geq 1$ and $\s_\ell = \min\{\ell+1,m\} \geq 0$, it holds that
 	\[
 	\sum_{\ell=0}^{p-1} h^{\g_\ell} \lesssim h, \qquad 	\sum_{\ell=0}^{p-1} h^{\s_\ell} \lesssim 1.
 	\]
 	the implicit constant in the inequality depending on $p$.
	\end{proof}